\documentclass[11pt, twoside, leqno]{article}

\usepackage{amssymb}
\usepackage{amsmath}
\usepackage{amsthm}
\usepackage{color}
\usepackage{mathrsfs}

\allowdisplaybreaks

\def\rn{{\mathbb{R}^n}}

\def\zz{{\mathbb Z}}

\def\nn{{\mathbb N}}

\def\cs{{\mathcal S}}

\def\fz{\infty }
\def\az{\alpha}

\def\dz{\delta}

\def\lz{\lambda}

\def\lf{\left}
\def\r{\right}

\def\hs{\hspace{0.25cm}}
\def\ls{\lesssim}
\def\gs{\gtrsim}

\def\noz{\nonumber}

\def\com{\complement}

\def\dis{\displaystyle}

\def\supp{\mathop\mathrm{\,supp\,}}

\def\lpq{{L^{p,q}(\rn)}}

\newtheorem{theorem}{Theorem}[section]
\newtheorem{lemma}[theorem]{Lemma}

\newtheorem{proposition}[theorem]{Proposition}

\theoremstyle{definition}
\newtheorem{remark}[theorem]{Remark}
\newtheorem{definition}[theorem]{Definition}
\renewcommand{\appendix}{\par
   \setcounter{section}{0}%
   \setcounter{subsection}{0}%
   \setcounter{subsubsection}{0}%
   \gdef\thesection{\@Alph\c@section}%
   \gdef\thesubsection{\@Alph\c@section.\@arabic\c@subsection}%
   \gdef\theHsection{\@Alph\c@section.}%
   \gdef\theHsubsection{\@Alph\c@section.\@arabic\c@subsection}%
   \csname appendixmore\endcsname
 }

\numberwithin{equation}{section}

\begin{document}

\arraycolsep=1pt

\title{\bf\Large Littlewood-Paley Characterizations
of Anisotropic Hardy-Lorentz Spaces
\footnotetext{\hspace{-0.35cm} 2010 {\it
Mathematics Subject Classification}. Primary  42B25;
Secondary 42B30, 46E30, 42B35, 30L99.
\endgraf {\it Key words and phrases.} Lorentz space,
anisotropic Hardy-Lorentz space, expansive matrix,
Calder\'{o}n reproducing formula,
Littlewood-Paley function.
\endgraf This project is supported by the National
Natural Science Foundation of China
(Grant Nos.~11571039, 11361020 and 11471042)
and the Specialized Research Fund for the Doctoral Program of Higher Education
of China (Grant No. 20120003110003).}}
\author{Jun Liu, Dachun Yang\,\footnote{Corresponding author}\ \
and Wen Yuan}
\date{}
\maketitle

\vspace{-0.8cm}

\begin{center}
\begin{minipage}{13cm}
{\small {\bf Abstract}\quad
Let $p\in(0,1]$, $q\in(0,\infty]$ and  $A$ be a general
expansive matrix on $\mathbb{R}^n$.
Let $H^{p,q}_A(\mathbb{R}^n)$ be the anisotropic Hardy-Lorentz spaces
associated with $A$ defined via the
non-tangential grand maximal function.
In this article, the authors
characterize $H^{p,q}_A(\mathbb{R}^n)$
in terms of the Lusin-area function,
the Littlewood-Paley $g$-function or the Littlewood-Paley $g_\lambda^*$-function
via first establishing an anisotropic Fefferman-Stein vector-valued
inequality in the Lorentz space $L^{p,q}(\mathbb{R}^n)$.
All these characterizations are
new even for the classical isotropic Hardy-Lorentz spaces on $\mathbb{R}^n$.
Moreover, the range of $\lambda$ in the $g_\lambda^*$-function
characterization of $H^{p,q}_A(\mathbb{R}^n)$ coincides with
the best known one in the classical Hardy space $H^p(\mathbb{R}^n)$ or in the anisotropic
Hardy space $H^p_A(\mathbb{R}^n)$.}
\end{minipage}
\end{center}

\section{Introduction\label{s1}}

\hskip\parindent
It is well known that $H^p(\rn)$ is a good substitute of the Lebesgue
space $L^p(\rn)$ with $p\in(0,1]$, particularly, in the study for the
boundedness of maximal functions and singular integral operators.
Moreover, when studying the boundedness of these operators in the
critical case, the weak Hardy space $H^{p,\fz}(\rn)$ naturally appears
and have been proved to be a good substitute of $H^p(\rn)$.
For example, it is known that, if $T$ is a $\dz$-type Calder\'on-Zygmund
operator with $\dz\in(0,1]$ and $T^*(1)=0$,
where $T^*$ denotes the \emph{adjoint operator} of $T$, then
$T$ is bounded on $H^p(\rn)$ for all
$p\in (\frac{n}{n+\dz},\,1]$ (see, for example, \cite{am86}),
but $T$ is not bounded on $H^{\frac{n}{n+\dz}}(\rn)$;
while Liu \cite{l91} proved that $T$ is
bounded from $H^{\frac{n}{n+\dz}}(\rn)$ to
$WH^{\frac{n}{n+\dz}}(\rn)$.
Liu \cite{l91} also obtained the $\fz$-atomic decomposition
of $WH^p(\rn)$ for all $p\in(0,1]$.
Recall that the weak
Hardy spaces $H^{p,\fz}(\rn)$ with $p\in(0,1)$ were first introduced by
Fefferman, Rivi\'{e}re and Sagher \cite{frs74} in 1974, which naturally appears
as the intermediate spaces of Hardy spaces $H^p_A(\mathbb{R}^n)$ with $p\in(0,1]$ via
the real interpolation. Later on, the weak Hardy space
$H^{1,\fz}(\rn)$ was introduced by Fefferman and Soria \cite{fs87}
to find out the biggest
space from which the Hilbert transform
is bounded to the weak Lebesgue space $L^{1,\fz}(\rn)$,
meanwhile, they also obtained the $\fz$-atomic decomposition of $H^{1,\fz}(\rn)$
and the boundedness of some Calder\'on-Zygmund operators from
$H^{1,\fz}(\rn)$ to $L^{1,\fz}(\rn)$.
In 1994, \'{A}lvarez \cite{a94} studied the Calder\'on-Zygmund
theory related to $H^{p,\fz}(\rn)$ with $p\in(0,1]$.
Nowadays, it is well known that the weak Hardy spaces
$H^{p,\fz}(\rn)$, with $p\in(0,1]$, play
a key role when studying the boundedness of operators in the critical case;
see, for example, \cite{a94,a98,dl03,g92,dls06,w13,dlq07}.
Moreover, the weak Hardy spaces $H^{p,\fz}(\rn)$
are also known as special cases of the Hardy-Lorentz spaces $H^{p,q}(\rn)$
which, when $p=1$ and $q\in(1,\fz)$, were introduced and studied
by Parilov \cite{p05}. For the full
range $p\in(0,1]$ and $q\in(0,\fz]$, the Hardy-Lorentz spaces $H^{p,q}(\rn)$
were investigated by Abu-Shammala and Torchinsky \cite{wa07}, and their $\fz$-atomic
characterizations, real interpolation properties over parameter $q$, and
the boundedness of singular integrals and some other operators on these spaces
were also prensented.
In 2010, Almeida and Caetano \cite{ac10} studied the generalized Hardy spaces,
which include the classical Hardy-Lorentz spaces $H^{p,q}(\rn)$ (see
\cite{wa07}) as special cases, in which, they obtained
some maximal characterizations and real interpolation results
of these generalized Hardy spaces
and, as applications,
they proved the boundedness of some classical operators in this generalized setting.

The Lorentz spaces were originally studied by Lorentz \cite{l50,l51}
in the early 1950's. As a generalization of $L^p(\rn)$, Lorentz spaces
are known as intermediate spaces of Lebesgue spaces
in the real interpolation method; see \cite{c64,lp64,p63}. For a  systematic
treatment of Lorentz spaces and their dual spaces, we refer the reader to
Hunt \cite{h66}, Cwikel \cite{c75} and Cwikel and Fefferman
\cite{cf80,cf84}; see also \cite{bs88,bl76,lg08,s76,sw71}. It is well known that,
due to their fine structures, Lorentz spaces play an irreplaceable role in the study
on various critical or endpoint analysis problems from many different research fields
and there exists a lot of literatures on this subject, here we only mention
several recent papers from harmonic analysis (see, for example,
\cite{osttw12,mtt03,st01,tw01}) and partial differential equations (see, for example,
\cite{iiy15,mr15,p15}).

With the enlightening work of Stein and Weiss \cite{sw60} on systems of conjugate harmonic functions,
higher-dimensional extensions of Hardy spaces naturally appear.
At the same time, a series of characterizations of Hardy spaces were obtained one after another;
see \cite{bgs71,l95,s93}.
Recall that the original work of the Littlewood-Paley theory should be owned
to Littlewood and Paley \cite{lp31}.
Moreover, the Littlewood-Paley theory of Hardy spaces
were investigated by Calder\'{o}n \cite{c65} and Fefferman and Stein \cite{fs72}.
In addition, the Littlewood-Paley theory of other useful function spaces, for example,
various forms of the Lipschitz spaces, the space BMO$(\rn)$ and Sobolev spaces,
has also been well developed, which provides one of the most successful unifying
perspectives on these spaces (see \cite{fjw91}).

On the other hand, from  1970's, there has been an increasing interesting in extending classical
function spaces arising in harmonic analysis from Euclidean spaces to
anisotropic settings and some other domains; see, for example,
\cite{ct75,ct77,fs82,st87,t83,t92}. The study of function spaces on $\rn$
associated with anisotropic dilations dates from these celebrated
articles \cite{c77,ct75,ct77} of Calder\'{o}n and Torchinsky on anisotropic Hardy spaces.
Since then, the theory of  anisotropic function spaces  was well developed
by many authors; see, for example, \cite{fs82,s93,t83}. In 2003, Bownik
\cite{mb03} introduced and investigated the anisotropic Hardy spaces
associated with general expansive matrixes, which were extended to the
weighted setting in \cite{blyz08}. For further efforts of function
spaces and related operators on the anisotropic Euclidean spaces, we refer the reader to
\cite{mb07,bh06,blyz08,dl08,lby14,lbyy12,lbyy14,t06}.

Moreover, the authors \cite{lyy15} introduced the anisotropic Hardy-Lorentz spaces
$H^{p,q}_A(\rn)$ and obtained some characterizations of these spaces;
for example, characterizations in terms of the atoms or the molecules,
the radial or the non-tangential maximal functions or the finite atomic decomposition
and, also considered some interpolation
properties of the anisotropic Hardy-Lorentz spaces
$H^{p,q}_A(\rn)$ via the real method and the boundedness of some classical
operators in this anisotropic setting. To further complete the theory of the anisotropic
Hardy-Lorentz spaces, in this article, we establish
the characterizations of $H^{p,q}_A(\rn)$
via Littlewood-Paley functions including the Lusin-area function,
the Littlewood-Paley $g$-function or the Littlewood-Paley $g_\lambda^*$-function.

To be precise, this article is organized as follows.

In Section \ref{s2}, we first present some basic notions and notation
used in this article, including Lorentz spaces and  their properties
and also some known facts on expansive matrixes in \cite{mb03}.
Then we recall the definitions of the anisotropic Hardy-Lorentz spaces
via non-tangential grand maximal functions
(denoted by $H^{p,q}_A(\rn)$) and their atomic variants
(denoted by $H^{p,r,s,q}_A(\rn)$). Moreover, we
state the main results of this article, namely,
the characterizations of $H^{p,q}_A(\rn)$
in terms of the Lusin-area function (see Theorem \ref{fivet1} below),
the Littlewood-Paley $g$-function (see Theorem \ref{fivet2} below)
or the Littlewood-Paley $g_\lambda^*$-function (see Theorem \ref{fivet3} below).
We point out that all these characterizations are
new even for the classical isotropic Hardy-Lorentz spaces on $\rn$.

In Section \ref{s5}, by the anisotropic Calder\'{o}n reproducing formula and
the method used in the proof of the atomic or the molecular characterizations
of $H^{p,q}_A(\rn)$ (see \cite{lyy15}),
we give out the proof of Theorem \ref{fivet1}.
We point out that,
when we decompose a distribution into a sum of atoms, the dual method used
in estimating
each atom in the classic case does not work any more in the present setting.
Instead,
we use a method from Fefferman \cite{f88} to
obtain a subtle estimate (see \eqref{five27} below).

In Section \ref{s3},
via the above Lusin area function characterizations
(namely, Theorem \ref{fivet1}), we first show Theorem \ref{fivet2}.
To this end, via borrowing some ideas from \cite[Theorem 1]{ycp05},
we establish an anisotropic Fefferman-Stein vector-valued inequality in
Lorentz spaces
$L^{p,q}(\rn)$ (see Lemma \ref{fivel6} below) which plays a key role
in the proof of Theorem \ref{fivet2} and is of independent interest.
Besides this,
we also employ the discrete Calder\'{o}n reproducing formula from
\cite[Lemma 3.2]{lfy14}, which is an anisotropic version of
\cite[Theorem 6.16]{fjw91}, and some auxiliary inequalities (see Lemmas \ref{fivel7}
and \ref{fivel8} below).
The proof of Lemma \ref{fivel8} borrows some ideas from
\cite[Lemma 3.3]{bh06}.
We also point out that the method used in the proof of Theorem \ref{fivet2}
is different from that used by Liang et al. in
the proof of \cite[Theorem 4.4]{lhy12}, in which a subtle pointwise upper estimate
via the vector-valued
Hardy-Littlewood maximal function (see \cite[(4.2)]{lhy12}) plays a key role.
Moreover, motivated by the proof of \cite[Theorem 3.9]{lfy14},
together with using some ideas from Folland and Stein
\cite[Theorem 7.1]{fs82} and Aguilera and Segovia \cite[Theorem 1]{as77}, we further prove
Theorem \ref{fivet3} for all $\lambda\in(2/p,\fz)$ in this section. To this end, we first prove
that the $L^{p,q}(\rn)$ quasi-norm of the variant of the anisotropic Lusin area
function $S_{k_0}(f)$ can be controlled by the $L^{p,q}(\rn)$ quasi-norm of the Lusin area
function $S(f)$ for all $k_0\in\nn$ and $f\in L^{p,q}(\rn)$ (see Lemma \ref{fivel11} below),
which is a key technique used in the proof of Theorem \ref{fivet3}. We point
out that the range of $\lambda$ in Theorem \ref{fivet3} coincides with the 
best known one in the classical Hardy space
$H^p(\rn)$ or in the anisotropic Hardy space $H^p_A(\rn)$.

Finally, we make some conventions on notation. Throughout this article, we always let
$\mathbb{N}:=\{1,2,\ldots\}$ and $\mathbb{Z}_+:=\{0\}\cup\mathbb{N}$.
For any multi-index
$\beta:=(\beta_1,\ldots,\beta_n)\in\mathbb{Z}_+^n$,
let
$|\beta|:=\beta_1+\cdots+\beta_n\ \mathrm{and}\
\partial^\beta:=(\frac\partial{\partial x_1})^{\beta_1}
\cdots(\frac\partial{\partial x_n})^{\beta_n}$.
The symbol $C$ denotes a \emph{positive constant} which is independent of the main
parameters, but its value may change from line to line.
\emph{Constants with subscripts}, such as $C_1$, are the same
in different occurrences. We also use $C_{(\alpha,\beta,\ldots)}$
to denote a positive constant depending on the indicated parameters
$\alpha$, $\beta,\ldots$.
Moreover,
we use $f\ls g$ to denote $f\leq Cg$ and, if $f\ls g\ls f$,
we then write $f\sim g$. For every index $r\in[1,\fz]$, we use $r'$ to
denote its \emph{conjugate index},
namely, $1/r+1/r'=1$. In addition, for any set $F\subset\rn$, we denote
by $F^\complement$ the
set $\rn\setminus F$,
by $\chi_F$ its \emph{characteristic function}, and by $\sharp F$ the cardinality of $F$.
The symbol $\lfloor s\rfloor$, for any $s\in\mathbb{R}$, denotes the maximal integer
not larger than $s$.

\section{Main results \label{s2}}

\hskip\parindent
In this section, we recall the notion of the anisotropic Hardy-Lorentz space
defined in \cite{lyy15}, and then state our main results.

We begin with the definition of Lorentz spaces. Let $p\in(0,\fz)$
and $q\in(0,\fz]$. The \emph{Lorentz space} $L^{p,q}(\rn)$
is defined to be the set of all measurable functions $f$ satisfying that
$\|f\|_{L^{p,q}(\rn)}<\fz$, where the quasi-norm
\begin{align*}
\|f\|_{L^{p,q}(\rn)}:=\left\{
\begin{array}{cl}
&\lf\{\dis\frac{q}{p}\dis\int_0^\fz\lf
[t^{\frac{1}{p}}{f^\ast(t)}\r]
^q\frac{dt}{t}\r\}^{\frac{1}{q}}
\ \ \ {\rm when}\ \ \ q\in(0,\fz),\\
&\dis\sup_{t\in(0,\fz)}\lf[t^
\frac{1}{p}f^\ast(t)\r]\ \ \ \ \ \ \
\ \ \ \ \ \ \ \,{\rm when}\ \ \ q=\infty,
\end{array}\r.
\end{align*}
and $f^\ast$ denotes the \emph{non-increasing rearrangement}
of $f$, namely,
\begin{align*}
f^\ast(t):=\{\alpha\in(0,\fz):\
d_f(\alpha)\leq t\}, \quad t\in(0,\fz).
\end{align*}
Here and hereafter, for any $\az\in(0,\fz)$,
$d_f(\alpha):=|\{x\in\rn:\ |f(x)|>\alpha\}|$.
It is well known that, if $q\in(0,\fz)$,
\begin{equation}\label{se6}
\|f\|_{L^{p,q}(\rn)}
\sim\lf\{\int_0^\fz\alpha^{q-1}\lf[d_f(\alpha)\r]
^{\frac{q}{p}}\,d\az\r\}^{\frac{1}{q}}
\sim \lf\{\sum_{k\in\mathbb{Z}}\lf[2^k
\lf\{d_f(2^k)\r\}^{\frac{1}{p}}\r]^q\r\}^{\frac{1}{q}}
\end{equation}
and
\begin{equation}\label{se7}
\|f\|_{L^{p,\infty}(\rn)}
\sim\sup_{\alpha\in(0,\fz)}
\lf\{\alpha\lf[d_f(\alpha)\r]^{\frac1p}\r\}
\sim\sup_{k\in\mathbb{Z}}
\lf\{2^k\lf[d_f(2^k)\r]^{\frac{1}{p}}\r\};
\end{equation}
see \cite{lg08}. By
\cite[Remark\ 1.4.7]{lg08}, for any $p,\,r\in(0,\fz)$, $q\in(0,\fz]$ and all measurable
functions $g$, we know that
\begin{align}\label{se22}
\lf\||g|^r\r\|_{L^{p,q}(\rn)}
=\lf\|g\r\|_{L^{pr,qr}(\rn)}^r.
\end{align}

Now let us recall the notion of
expansive matrices in \cite{mb03}.

\begin{definition}\label{d-em}
An $n\times n$ real matrix $A$
is called an \emph{expansive matrix}
(for short, a \emph{dilation})
if $\min_{\lz\in\sigma(A)}|\lz|>1$,
here and hereafter, $\sigma(A)$ denotes the set of
all eigenvalues of $A$.
\end{definition}

Throughout this article,
$A$ always denotes a fixed dilation
and $b:=|\det A|$. By \cite[p.\,6, (2.7)]{mb03},
we have $b\in(1,\fz)$.
Let $\lambda_-$ and $\lambda_+$
be two \emph{positive numbers} satisfying that
$$1<\lambda_-<\min\{|\lambda|:\
\lambda\in\sigma(A)\}
\leq\max\{|\lambda|:\
\lambda\in\sigma(A)\}<\lambda_+.$$
In the case when $A$ is diagonalizable over
$\mathbb{C}$, we can even take
$\lambda_-:=\min\{|\lambda|:\
\lambda\in\sigma(A)\}$
and
$\lambda_+:=\max\{|\lambda|:\
\lambda\in\sigma(A)\}$.
Otherwise, we need to choose them
sufficiently close to these equalities
according to what we need in our arguments.
Then, by \cite[p.\,5 (2.1) and (2.2)]{mb03},
there exists a positive constant
$C$, independent of $x$ and $j$, such that, for all $x\in\rn$,
when $j\in\mathbb{Z}_+$,
\begin{align*}
C^{-1}\lf(\lambda_-\r)^j|x|\leq|A^jx|
\leq C\lf(\lambda_+\r)^j|x|
\end{align*}
and, when $j\in\mathbb{Z}\setminus\mathbb{Z}_+$,
\begin{align}\label{se19}
C^{-1}\lf(\lambda_+\r)^j|x|\leq|A^jx|
\leq C\lf(\lambda_-\r)^j|x|.
\end{align}

It was proved in \cite[p.\,5, Lemma 2.2]{mb03} that,
for a given dilation $A$, there exists an open
ellipsoid $\Delta$ and $r\in(1,\infty)$ such that
$\Delta\subset r\Delta\subset A\Delta$, and one
can additionally assume that $|\Delta|=1$,
where $|\Delta|$ denotes the \emph{n-}dimensional
Lebesgue measure of the set $\Delta$.
Let $B_k:=A^k\Delta$ for all $k\in\zz$.
An ellipsoid $x+B_k$ for some $x\in\rn$ and $k\in\mathbb{Z}$
is called a \emph{dilated ball}. Let $\mathfrak{B}$ be the set of all such
dilated balls, namely,
\begin{align}\label{se14}
\mathfrak{B}:=\lf\{x+B_k:\ x\in\rn,\ k\in\mathbb{Z}\r\}.
\end{align}
Then $B_k$ is open,
$B_k\subset rB_k\subset B_{k+1}$
and $|B_k|=b^k$.
Throughout this article, let $\tau$ be
the \emph{minimal integer} such that
$r^\tau\geq2$. Then, for all $k\in\mathbb{Z}$,
it holds true that
\begin{align}\label{se1}
B_k+B_k\subset B_{k+\tau},
\end{align}
\begin{align}\label{se2}
B_k+(B_{k+\tau})^\complement
\subset (B_k)^\complement,
\end{align}
where $E+F$ denotes the algebraic sum
$\{x+y:\ x\in E,\,y\in F\}$ of sets
$E,\,F\subset\mathbb{R^\emph{n}}$.

The notion of the homogeneous quasi-norm
induced by $A$ was introduced in
\cite[p.\,6, Definition 2.3]{mb03} as follows.

\begin{definition}\label{sd2}
A \emph{homogeneous quasi-norm} associated
with a dilation $A$ is a measurable mapping
$\rho:\ \rn \rightarrow [0,\infty]$ satisfying that

(i) $\rho(x)=0 \Longleftrightarrow x=\vec0_n$,
here and hereafter, $\vec0_n:=(0,\ldots,0)\in\rn$;

(ii) $\rho(Ax)=b\rho(x)$ for all $x\in\rn$;

(iii) $\rho(x+y)\leq H[\rho(x)+\rho(y)]$
for all $x,\,y\in\rn$, where $H\in[1,\fz)$ is a constant
independent of $x$ and $y$.
 \end{definition}

In the standard dyadic case
$A:=2{\rm I}_{n\times n}$, $\rho(x):=|x|^n$
for all $x\in\rn$ is an example of the
homogeneous quasi-norm associated
with $A$, here and hereafter,
${\rm I}_{n\times n}$ denotes the $n\times n$
\emph{unit matrix} and $|\cdot|$ denotes the
Euclidean norm in $\rn$. It was proved in
\cite[p.\,6, Lemma 2.4]{mb03} that all homogeneous
quasi-norms associated with $A$ are equivalent.
Therefore, for a given dilation $A$, in what follows,
we always use the
\emph{step homogeneous quasi-norm} $\rho$
defined by setting, for all $x\in\rn$,
\begin{equation*}
\rho(x):=\left\{
\begin{array}{cl}
b^j&\ \ \ \ {\rm when}\ \ \ x\in
B_{j+1}\backslash B_j,\\
0&\ \ \ \ {\rm when}\ \ \ x=0
\end{array}\r.
\end{equation*}
for convenience.
Obviously, for all $k\in\zz$,
$B_k=\{x\in\rn:\ \rho(x)<b^k\}$.
Moreover, $(\rn,\,\rho,\,dx)$ is a space
of homogeneous type in the sense of
Coifman and Weiss \cite{cw71,cw77}, here and hereafter,
$dx$ denotes the \emph{n-}dimensional
Lebesgue measure.

Denote the space of all Schwartz functions by
$\cs(\rn)$, namely, the set of all $C^\infty(\rn)$
functions $\varphi$ satisfying that, for every integer $\ell\in\zz_+$ and
multi-index $\alpha$,
$$\|\varphi\|_{\alpha,\ell}:=
\sup_{x\in\rn}[\rho(x)]^\ell
\lf|\partial^\alpha\varphi(x)\r|<\infty.$$
The dual space of $\cs(\rn)$, namely, the space of all
tempered distributions on $\rn$ equipped with the weak-$\ast$ topology, is denoted by
$\cs'(\rn)$. For any $N\in\mathbb{Z}_+$, let
$$\cs_N(\rn):=\{\varphi\in\cs(\rn):\
\|\varphi\|_{\alpha,\ell}\leq1,\
|\alpha|\leq N\ ,\ \ell\leq N\};$$
equivalently,
\begin{align*}
\ \ \ \varphi\in\cs_N(\rn)\Longleftrightarrow
\|\varphi\|_{\cs_N(\rn)}:=\sup_{|\alpha|\leq N}
\sup_{x\in\rn}\lf[\lf|\partial^\alpha
\varphi(x)\r|\max\lf\{1,\lf[
\rho(x)\r]^N\r\}\r]\leq1.
\end{align*}
Throughout this article, for
$\varphi\in\cs(\rn)$ and $k\in\mathbb{Z}$, let
$\varphi_k(\cdot):=b^{-k}\varphi(A^{-k}\cdot)$.

\begin{definition}\label{d-mf}
Let $\varphi\in\cs(\rn)$ and $f\in\cs'(\rn)$. The
\emph{non-tangential maximal function} $M_\varphi(f)$ of
$f$ with respect to $\varphi$ is defined as
\begin{align*}
M_\varphi(f)(x):= \sup_{y\in x+B_k,
k\in\mathbb{Z}}|f\ast\varphi_k(y)|,
\ \ \ \ \ \forall\ x\in\rn.
\end{align*}
Moreover, for $N\in\mathbb{N}$, the
\emph{non-tangential grand maximal function} $M_N(f)$
of $f$ is defined as
\begin{equation}\label{se8}
M_N(f)(x):=\sup_{\varphi\in\cs_N(\rn)}
M_\varphi(f)(x),\ \ \ \ \ \forall\ x\in\rn.
\end{equation}
\end{definition}

The following Proposition \ref{tl4}
is just \cite[p.\,13, Theorem 3.6]{mb03}.

\begin{proposition}\label{tl4}
For any $s\in(1,\fz)$, let
$$\mathcal{F}:=
\lf\{\varphi\in L^\infty(\rn):\ |\varphi(x)|
\leq[1+\rho(x)]^{-s},\ x\in\rn\r\}.$$
For $p\in[1,\fz]$ and $f\in L^p(\rn)$,
the maximal function associated with $\mathcal{F}$,
$M_{\mathcal{F}}$, is defined as
\begin{align*}
M_{\mathcal{F}}(f)(x):=
\sup_{\varphi\in\mathcal{F}}M_\varphi(f)(x),
\ \ \ \ \ \ \forall\ x\in\rn.
\end{align*}
Then there exists a positive constant
$C_{(s)}$, depending on $s$, such that,
for all $\lz\in(0,\fz)$ and $f\in L^1(\rn)$,
\begin{align}\label{te56}
\lf|\lf\{x\in\rn:\ M_{\mathcal{F}}(f)(x)>\lambda\r\}\r|
\leq C_{(s)}\|f\|_{L^1(\rn)}/\lambda
\end{align}
and, for all $p\in(1,\fz]$, there exists a positive constant
$C_{(s,p)}$, depending on $s$ and $p$, such that,
for all $f\in L^p(\rn)$,
\begin{align}\label{te57}
\|M_{\mathcal{F}}(f)\|_{L^p(\rn)}
\leq C_{(s,p)}\|f\|_{L^p(\rn)}.
\end{align}
\end{proposition}

\begin{remark}\label{tr1}
Obviously, by Proposition \ref{tl4}, we know that the non-tangential grand
maximal function $M_N(f)$ in \eqref{se8} and the
Hardy-Littlewood maximal function $M_{{\rm HL}}(f)$, defined by setting,
for all $f\in L^1_{{\rm loc}}(\rn)$ and $x\in\rn$,
\begin{align}\label{te58}
M_{{\rm HL}}(f)(x):=\sup_{k\in\mathbb{Z}}
\sup_{y\in x+B_k}\frac1{|B_k|}
\int_{y+B_k}|f(z)|\,dz=\sup_{x\in B\in\mathfrak{B}}
\frac1{|B|}\int_B|f(z)|\,dz,
\end{align}
where $\mathfrak{B}$ is as in \eqref{se14}, satisfy \eqref{te56} and \eqref{te57}.
\end{remark}

The following anisotropic Hardy-Lorentz space and
its atomic variant were
introduced in \cite{lyy15}.

\begin{definition}\label{d-ahls}
Let $p\in(0,\fz),\,q\in(0,\fz]$ and
\begin{align*}
N_{(p)}:=\left\{
\begin{array}{rl}
&\lf\lfloor\lf(\dfrac1p-1\r)\dfrac{\ln b}{\ln
\lambda_-}\r\rfloor+2\hspace{1cm} {\rm when}\ p\in(0,1],\\
&2\hspace{4.44cm} {\rm when}\ p\in(1,\fz).
\end{array}\r.
\end{align*}
 For any $N\in\nn\cap(N_{(p)},\fz)$, the
\emph{anisotropic Hardy-Lorentz space} $H^{p,q}_A(\rn)$ is defined by
\begin{equation*}
H_A^{p,q}(\rn)
:=\lf\{f\in\cs'(\rn):\ M_N(f)\in\lpq\r\}
\end{equation*}
and, for any $f\in H^{p,q}_A(\rn)$, let
$\|f\|_{H^{p,q}_A(\rn)}
:=\| M_N(f)\|_{L^{p,q}(\rn)}$.
\end{definition}

\begin{definition}\label{d-ats}
(i)
An anisotropic triplet $(p,r,s)$ is said to be \emph{admissible}
if $p\in(0,1]$, $r\in(1,\fz]$ and $s\in\mathbb{N}$ with
$s\geq\lfloor(1/p-1)\ln b/\ln\lambda_-\rfloor$.
For an admissible anisotropic triplet $(p,r,s)$, a measurable
function $a$ on $\rn$ is called an \emph{anisotropic $(p,r,s)$-atom}
if
\begin{enumerate}
\item[(a)] ${\rm \supp}a \subset B,\ {\rm where}\
B\in\mathfrak{B}$ and $\mathfrak{B}$ is as in \eqref{se14};

\item[(b)] $\|a\|_{L^r(\rn)}\leq |B|^{1/r-1/p}$;

\item[(c)] $\int_\rn a(x)x^\alpha\,dx =0$ for any
$\alpha\in\zz_+^n$ with $|\alpha|\leq s$.
\end{enumerate}

(ii) For an anisotropic triplet $(p,r,s)$,
$q\in(0,\fz]$ and a dilation
$A$, the \emph{anisotropic atomic Hardy-Lorentz space}
$H_A^{p,r,s,q}(\rn)$ is defined to be the set of all distributions
$f\in\cs'(\rn)$ satisfying that there exist a sequence
$\{a_i^k\}_{i\in\mathbb{N},\,k\in\mathbb{Z}}$ of $(p,r,s)$-atoms,
respectively, supported on
$\{x_i^k+B_i^k\}_{i\in\mathbb{N},\,k\in\mathbb{Z}}\subset\mathfrak{B}$,
and a positive constant $\widetilde{C}$ such that
$\sum_{i\in\mathbb{N}}\chi_{x_i^k+B_i^k}(x)\leq \widetilde{C}$
for all $x\in\rn$ and $k\in\mathbb{Z}$, and
$f=\sum_{k\in\mathbb{Z}}\sum_{i\in\mathbb{N}}\lambda_i^ka_i^k$
in $\cs'(\rn)$, where $\lambda_i^k\sim2^k|B_i^k|^{1/p}$ for all
$k\in\mathbb{Z}$ and $i\in\mathbb{N}$ with the implicit equivalent
positive constants independent of $k$ and $i$.

Moreover, for all $f\in H_A^{p,r,s,q}(\rn)$, define
\begin{align*}
\|f\|_{H_A^{p,r,s,q}(\rn)}:=
{\rm inf}\lf\{\lf[\sum_{k\in\zz}\lf(\sum_{i\in\nn}|\lambda_i^k|^p\r)^
{\frac qp}\r]^{\frac 1q}:\ f=\sum_{k\in\mathbb{Z}}
\sum_{i\in\mathbb{N}}\lambda_i^ka_i^k\r\}
\end{align*}
with the usual modification made when $q=\fz$, where the
infimum is taken over all decompositions of $f$ as above.
\end{definition}

The following decomposition characterizations of $H_A^{p,q}(\rn)$
is just \cite[Theorem 3.6]{lyy15}.
\begin{lemma}\label{sl1}
Let $(p,r,s)$ be an admissible anisotropic triplet as in
Definition \ref{d-ats}(i), $q\in(0,\fz]$ and $N\in\nn\cap(N_{(p)},\fz)$. Then
$H^{p,q}_A(\rn)=H_A^{p,r,s,q}(\rn)$ with equivalent quasi-norms.
\end{lemma}

\begin{remark}\label{sr1}
From Lemma \ref{sl1}, it follows that the space $H^{p,q}_A(\rn)$
is independent of the choice of $N$ as long as
$N\in\nn\cap(N_{(p)},\fz)$.
\end{remark}

\begin{definition}\label{fived1}
Suppose that $p\in(0,1]$ and $\varphi\in\cs(\rn)$ satisfies
$\int_{\rn}\varphi(x)x^\alpha\,dx=0$
for all multi-indices $\alpha\in\zz_+^n$ with
$|\alpha|\leq s$, where $s\in\mathbb{N}$ and
$s\geq\lfloor(1/p-1)\ln b/\ln\lambda_-\rfloor$. For all $f\in\cs'(\rn)$ and
$\lambda\in(0,\fz)$, the \emph{anisotropic Lusin area function} $S(f)$,
the \emph{Littlewood-Paley} $g$-\emph{function} $g(f)$ and
the \emph{Littlewood-Paley} $g_\lambda^\ast$-\emph{function} $g_\lambda^\ast(f)$
are defined, respectively, by setting, for all $x\in\rn$,

\begin{align}\label{five1}
S(f)(x):=\lf[\sum_{k\in\mathbb{Z}}b^{-k}\int_{x+B_k}
\lf|f\ast\varphi_{k}(y)\r|^2\,dy\r]^{1/2},
\end{align}
\begin{align*}
g(f)(x):=\lf[\sum_{k\in\mathbb{Z}}
\lf|f\ast\varphi_{k}(x)\r|^2\r]^{1/2}
\end{align*}
and
\begin{align*}
g_\lambda^\ast(f)(x):=
\lf\{\sum_{k\in\mathbb{Z}}b^{-k}\int_{\rn}
\lf[\frac{b^k}{b^k+\rho(x-y)}\r]^\lambda
\lf|f\ast\varphi_{k}(y)\r|^2\,dy\r\}^{1/2}.
\end{align*}
\end{definition}

Recall that a distribution $f\in\cs'(\rn)$ is said to
\emph{vanish weakly at infinity} if, for any $\psi\in\cs(\rn)$,
$f\ast\psi_{k}\rightarrow0$ in $\cs'(\rn)$ as $k\rightarrow\fz$.
Denote by $\cs'_0(\rn)$ the set of all $f\in\cs'(\rn)$
vanishing weakly at infinity.

The following Theorems \ref{fivet1} through \ref{fivet3} are
the main results of this article, which characterize the space $H_A^{p,q}(\rn)$,
respectively, in terms of
the Lusin area function, the Littlewood-Paley $g$-function and
the Littlewood-Paley $g_\lambda^\ast$-function.

\begin{theorem}\label{fivet1}
Let $p\in(0,1]$ and $q\in(0,\fz]$.
Then $f\in H_A^{p,q}(\rn)$ if and only if $f\in\cs'_0(\rn)$
and $S(f)\in L^{p,q}(\rn)$.
Moreover, there exists a positive constant $C_1$ such that,
for all $f\in H_A^{p,q}(\rn)$,
$$\frac1{C_1}\lf\|S(f)\r\|_{L^{p,q}(\rn)}\leq\|f\|_{H_A^{p,q}(\rn)}
\leq C_1\lf\|S(f)\r\|_{L^{p,q}(\rn)}.$$
\end{theorem}

\begin{theorem}\label{fivet2}
Let $p\in(0,1]$ and $q\in(0,\fz]$.
Then $f\in H_A^{p,q}(\rn)$ if and only if
$f\in\cs'_0(\rn)$ and $g(f)\in L^{p,q}(\rn)$.
Moreover, there exists a positive constant
$C_2$ such that, for all $f\in H_A^{p,q}(\rn)$,
$$\frac1{C_2}\lf\|g(f)\r\|_{L^{p,q}(\rn)}
\leq\|f\|_{H_A^{p,q}(\rn)}
\leq C_2\lf\|g(f)\r\|_{L^{p,q}(\rn)}.$$
\end{theorem}

\begin{theorem}\label{fivet3}
Let $p\in(0,1],\,q\in(0,\fz]$ and
$\lambda\in(2/p,\fz)$.
Then $f\in H_A^{p,q}(\rn)$ if and only if $f\in\cs'_0(\rn)$
and $g_{\lambda}^*(f)\in L^{p,q}(\rn)$.
Moreover, there exists a positive constant $C_3$ such that,
for all $f\in H_A^{p,q}(\rn)$,
$$\frac1{C_3}\lf\|g_\lambda^*(f)\r\|_{L^{p,q}(\rn)}
\leq\|f\|_{H_A^{p,q}(\rn)}
\leq C_3\lf\|g_\lambda^*(f)\r\|_{L^{p,q}(\rn)}.$$
\end{theorem}

\begin{remark}\label{sr2}
We point
out that the range of $\lambda$ in Theorem \ref{fivet3} coincides with the
best known one in the classical Hardy space
$H^p(\rn)$ or in the anisotropic Hardy space $H^p_A(\rn)$.
\end{remark}

\section{Proof of Theorem \ref{fivet1}}\label{s5}

\hskip\parindent
From \cite[Proposition 2.7]{lyy15}, we deduce that,
for all $p\in(0,\fz)$ and $q\in(0,\fz]$,
$H_A^{p,q}(\rn)\subset\cs'(\rn)$. Moreover,
we have the following useful property of $H_A^{p,q}(\rn)$.

\begin{proposition}\label{fivep1}
Let $p\in(0,1]$ and $q\in(0,\fz]$.
Then
$$H_A^{p,q}(\rn)\subset\cs'_0(\rn).$$
\end{proposition}

\begin{proof}
Observe that, for any
$\psi\in\cs(\rn),\,k\in\mathbb{Z},\,f\in H_A^{p,q}(\rn),\,x\in\rn$
and $y\in x+B_k$, $|f\ast\psi_k(x)|\ls M_N(f)(y)$.
Hence, there exists a positive constant $C_4$ such that
$$x+B_k\subset\lf\{y\in\rn:\ M_N(f)(y)>C_4|f\ast\psi_k(x)|\r\}.$$
By this and \eqref{se7}, we further have
\begin{align*}
|f\ast\psi_k(x)|
&=|B_k|^{-1/p}|B_k|^{1/p}|f\ast\psi_k(x)|\\
&\ls|B_k|^{-1/p}\lf|\lf\{y\in\rn:\ M_N(f)(y)>
C_4|f\ast\psi_k(x)|\r\}\r|^{1/p}|f\ast\psi_k(x)|\\
&\ls|B_k|^{-1/p}\|f\|_{H^{p,\fz}(\rn)}
\ls|B_k|^{-1/p}\|f\|_{H_A^{p,q}(\rn)}\rightarrow0
\end{align*}
as $k\rightarrow\fz$. Thus,
$f\in\cs'_0(\rn)$, which completes the proof of Proposition \ref{fivep1}.
\end{proof}

The following lemma is just \cite[Lemma 1.2]{wa07}.

\begin{lemma}\label{tl3}
Suppose that $p\in(0,\infty)$, $q\in(0,\fz]$,
$\{\mu_k\}_{k\in\mathbb{Z}}$ is a non-negative sequence
of complex numbers such that
$\{2^k\mu_k\}_{k\in\mathbb{Z}}\in\ell^q$ and
$\varphi$ is a non-negative function having the following
property: there exists $\delta\in(0,\min\{1, q/p\})$ such that,
for any $k_0\in\mathbb{N}$,
$\varphi\leq\psi_{k_0}+\eta_{k_0}$, where $\psi_{k_0}$
and $\eta_{k_0}$ are functions, depending on $k_0$ and satisfying
\begin{align*}
2^{k_0p}\lf[d_{\psi_{k_0}}(2^{k_0})\r]^\delta
\leq\widetilde{C}\sum_{k=-\infty}^{k_0-1}
\lf[2^k(\mu_k)^\delta\r]^p,\ \ \ 2^{k_0\delta p}
d_{\eta_{k_0}}(2^{k_0})\leq\widetilde{C}
\sum_{k=k_0}^\infty\lf[2^{k\delta}\mu_k\r]^p
\end{align*}
for some positive constant $\widetilde{C}$
independent of $k_0$. Then
$\varphi\in L^{p,q}(\rn)$ and
$$\|\varphi\|_{L^{p,q}(\rn)}\leq
C\|\{2^k\mu_k\}_{k\in\mathbb{Z}}\|_{\ell^q},$$
where $C$ is a positive constant independent of
$\varphi$ and $\{\mu_k\}_{k\in\mathbb{Z}}$.
\end{lemma}

The following lemma is just \cite[Lemma 2.3]{blyz10},
which is a slight modification of \cite[Theorem 11]{cm90}.

\begin{lemma}\label{fivel1}
Suppose that $A$ is a dilation on $\rn$. Then there exists a collection
$$\mathcal{Q}:=\{Q_\alpha^k\subset\rn:\ k\in\mathbb{Z},
\,\alpha\in I_k\}$$
of open subsets, where $I_k$ is certain index set, such that
\begin{enumerate}
\item[{\rm (i)}] $\lf|\rn\setminus\bigcup_{\alpha}Q_\alpha^k\r|=0$
for each fixed $k$ and
$Q_\alpha^k\cap Q_\beta^k=\emptyset$ if $\alpha\neq\beta$;
\item[{\rm(ii)}] for any $\alpha,\,\beta,\,k,\,\ell$ with $\ell\geq k$,
either $Q_\alpha^k\cap Q_\beta^\ell=\emptyset$ or
$Q_\alpha^\ell\subset Q_\beta^k$;
\item[{\rm(iii)}] for each $(\ell,\beta)$ and each $k<\ell$,
there exists a unique $\alpha$ such that
$Q_\beta^\ell\subset Q_\alpha^k$;
\item[{\rm(iv)}] there exist certain negative integer $v$
and positive integer $u$ such that, for all $Q_\alpha^k$
with $k\in\mathbb{Z}$ and $\alpha\in I_k$,
there exists $x_{Q_\alpha^k}\in Q_\alpha^k$
satisfying that, for all $x\in Q_\alpha^k$,
$$x_{Q_\alpha^k}+B_{vk-u}
\subset Q_\alpha^k\subset x+B_{vk+u}.$$
\end{enumerate}
\end{lemma}

In what follows, for convenience, we call
$\mathcal{Q}:=
\{Q_\alpha^k\}_{k\in\mathbb{Z},\,\alpha\in I_k}$
from Lemma \ref{fivel1} \emph{dyadic cubes} and
$k$ the \emph{level}, denoted by $\ell(Q_\alpha^k)$,
of the dyadic cube $Q_\alpha^k$
with $k\in\mathbb{Z}$ and $\alpha\in I_k$.

\begin{remark}\label{fiver1}
In the definition of $(p,r,s)$-atoms (see Definition \ref{d-ats}(i)),
if we replace dilated balls $\mathfrak{B}$ (see \eqref{se14}) by
dyadic cubes, from Lemma \ref{fivel1},
it follows that the corresponding anisotropic atomic Hardy-Lorentz space
coincides with the original one (see Definition \ref{d-ats}(ii))
in the sense of equivalent quasi-norms.
\end{remark}

The following  Calder\'{o}n reproducing formula is just
\cite[Proposition 2.14]{blyz10}.

\begin{lemma}\label{fivel2}
Let $s\in\mathbb{Z_+}$ and $A$ be a dilation on $\rn$.
Then there exist $\theta,\,\psi\in\cs(\rn)$ such that
\begin{enumerate}
\item[{\rm(i)}] $\supp\theta\subset B_0,
\,\int_{\rn}x^\gamma\theta(x)\,dx=0$ for all
$\gamma\in\zz_+^n$ with $|\gamma|\leq s,
\,\widehat{\theta}(\xi)\geq C$
for all $\xi\in\{x\in\rn:\ a\leq\rho(x)\leq b\}$,
where $0<a<b<1$ and $C$ are positive constants;
\item[{\rm(ii)}] $\supp \widehat{\psi}$
is compact and bounded away from the origin;
\item[{\rm(iii)}] $\sum_{j\in\mathbb{Z}}
\widehat{\psi}((A^\ast)^j\xi)\widehat{\theta}((A^\ast)^j\xi)=1$
for all $\xi\in\rn\setminus\{\vec{0}_n\}$,
where $A^\ast$ denotes the adjoint of $A$.
\end{enumerate}
Moreover, for all $f\in\cs'_0(\rn),\,f=
\sum_{j\in\mathbb{Z}}f\ast\psi_j\ast\theta_j$ in $\cs'(\rn)$.
\end{lemma}

Now we prove Theorem \ref{fivet1}.

\begin{proof}[Proof of Theorem \ref{fivet1}]
We first prove the necessity of Theorem \ref{fivet1}.
To this end,
assume that $f\in H_A^{p,q}(\rn)$.
By Proposition
\ref{fivep1}, $f\in\cs'_0(\rn)$.
It remains to show $S(f)\in L^{p,q}(\rn)$ and
$\lf\|S(f)\r\|_{L^{p,q}(\rn)}\ls\|f\|_{H_A^{p,q}(\rn)}$.
We prove this by three steps.

\emph{Step 1.}
By Lemma \ref{sl1} and Remark \ref{fiver1},
for $f\in H_A^{p,q}(\rn)=H_A^{p,r,s,q}(\rn)$,
we know that there exists a sequence
$\{a_i^k\}_{i\in\mathbb{N},\,k\in\mathbb{Z}}$ of $(p,r,s)$-atoms,
respectively, supported on
$\{Q_i^k\}_{i\in\mathbb{N},\,k\in\mathbb{Z}}\subset\mathcal{Q}$
such that $f=\sum_{k\in\mathbb{Z}}
\sum_{i\in\mathbb{N}}\lambda_i^ka_i^k$ in $\cs'(\rn)$,
$\lambda_i^k\sim2^k|Q_i^k|^{1/p}$
for all $k\in\mathbb{Z}$ and $i\in\mathbb{N}$,
$\sum_{i\in\mathbb{N}}\chi_{Q_i^k}(x)\ls1$
for all $k\in\mathbb{Z}$ and $x\in\rn$, and

\begin{align}\label{five61}
\|f\|_{H_A^{p,q}(\rn)}\sim\lf[\sum_{k\in\zz}
\lf(\sum_{i\in\nn}|\lambda_i^k|^p\r)^{\frac qp}\r]^{\frac 1q}.
\end{align}

\emph{Step 2.} In this step, we prove that, for all $j\in\mathbb{Z_+}$ and
\begin{align}\label{five21}
x\in U_j:=x_{Q}+(B_{v[\ell(Q)-j-1]+u+2\tau}
\setminus B_{v[\ell(Q)-j]+u+2\tau}),
\end{align}
it holds true that
\begin{align}\label{five7}
S(a)(x)\ls (1-vj)^{1/2}b^{\frac{1+\epsilon}pvj}
b^{-\frac{v\ell(Q)}r}\|a\|_{L^r(Q)}
\end{align}
for some $\epsilon\in(0,\fz)$, where $v,u$ are as in Lemma
\ref{fivel1} and $a$ is a $(p,r,s)$-atom
supported on a dyadic cube $Q$.

To this end,
for any $j\in\mathbb{Z_+}$ and $x\in U_j$, we write
\begin{align*}
\lf[S(a)(x)\r]^2=
\lf[\sum_{k>-v\ell(Q)-u}+\sum_{k\leq-v\ell(Q)-u}\r]b^k
\int_{x+B_{-k}}|a\ast\varphi_{-k}(y)|^2\,dy
=:{\rm II_1}+{\rm II_2}.
\end{align*}

For ${\rm II_1}$, by $-k<v\ell(Q)+u$ and \eqref{se2},
for all $y\in x+B_{-k}$, we know that
$$y\in x+B_{-k}\subset x_{Q}+\lf(B_{v[\ell(Q)-j]+u+2\tau}\r)
^\complement+B_{-k}
\subset  x_{Q}+\lf(B_{v[\ell(Q)-j]+u+\tau}\r)^\complement,$$
which, together with \eqref{se2}, further implies that, for all
$z\in Q\subset  x_{Q}+B_{v\ell(Q)+u}$,
$$y-z\in\lf(B_{v[\ell(Q)-j]+u}\r)^\complement,$$ namely,
\begin{align}\label{five9}
\rho(y-z)\geq b^{v[\ell(Q)-j]+u}.
\end{align}
Since $\varphi\in\cs(\rn)$, taking some $\widetilde{N}\in\mathbb{N}$ to be fixed later,
by \eqref{five9}, the H\"{o}lder inequality and Lemma \ref{fivel1}(iv),
we conclude that, for all $y\in x+B_{-k}$,
\begin{align}\label{five40}
|a\ast\varphi_{-k}(y)|
&\ls b^k\int_{Q}\lf|a(z)\varphi(A^k(y-z))\r|\,dz\\
&\ls b^k\int_{Q}|a(z)|
\frac1{\max\lf\{1,\lf[\rho(A^k(y-z))\r]^{\widetilde{N}+1}\r\}}\,dz\noz\\
&\ls b^{k-(\widetilde{N}+1)\{k+v[\ell(Q)-j]+u\}}\int_{Q}|a(z)|\,dz\noz\\
&\ls b^{k-(\widetilde{N}+1)\{k+v[\ell(Q)-j]+u\}}
b^{\frac{v\ell(Q)}{r'}}\|a\|_{L^r(Q)}\noz\\
&\sim b^{-\widetilde{N}\{k+v[\ell(Q)-j]\}}b^{vj}
b^{-\frac{v\ell(Q)}r}\|a\|_{L^r(Q)}.\noz
\end{align}
By
$s\geq\lfloor(1/p-1)\ln b/\ln\lambda_-\rfloor$,
choosing $\widetilde{N}\in\mathbb{N}$ such that
$b^{\widetilde{N}}>b(\lambda_-)^{s+1}>b^{1/p}$, it follows that
there exists a positive constant $\epsilon$ such that
$b^{\widetilde{N}}>b^{\frac{1+\epsilon}p}$. This, combined with \eqref{five40},
implies that
\begin{align}\label{five11}
{\rm II_1}&\ls\sum_{k>-v\ell(Q)-u}
b^{-2\widetilde{N}\{k+v[\ell(Q)-j]\}}b^{2vj}
b^{-\frac{2v\ell(Q)}r}\|a\|^2_{L^r(Q)}\\
&\ls b^{2\widetilde{N}vj}b^{-\frac{2v\ell(Q)}r}\|a\|
^2_{L^r(Q)}\noz\\
&\ls b^{2\frac{1+\epsilon}pvj}
b^{-\frac{2v\ell(Q)}r}\|a\|^2_{L^r(Q)}
\ls(1-vj)b^{2vj\frac{1+\epsilon}p}
b^{-\frac{2v\ell(Q)}r}\|a\|^2_{L^r(Q)},\noz
\end{align}
where the second and the last inequalities
follow from the fact $v\in(-\fz,0)$.

For ${\rm II_2}$, by
$A^k(z-x_Q)\in B_{k+v\ell(Q)+u}$ for $z\in Q$,
the fact that $k+v\ell(Q)+u\leq0$
and \eqref{se19}, we have
\begin{align}\label{five12}
|A^k(z-x_Q)|\leq(\lambda_-)^{k+v\ell(Q)+u}.
\end{align}
Moreover, for all $y\in x+B_{-k}$, if $v[\ell(Q)-j]+u>-k$,
by \eqref{se2}, we obtain
$$y\in x+B_{-k}\subset x_{Q}+
(B_{v[\ell(Q)-j]+u+2\tau})^\complement+B_{-k}\subset
x_{Q}+(B_{v[\ell(Q)-j]+u+\tau})^\complement.$$
Using this, for all $\zeta\in Q\subset x_Q+B_{v\ell(Q)+u}$,
by \eqref{se2} again, we further know that
$y-\zeta\in(B_{v[\ell(Q)-j]+u})^\complement$,
which implies that
$$\rho(y-\zeta)\geq b^{v[\ell(Q)-j]+u}$$
and hence
\begin{align}\label{five13}
\sup_{\zeta\in Q}\frac1
{\max\lf\{1,\lf[\rho(A^k(y-\zeta))\r]^{N+1}\r\}}
\leq\min\lf\{1,b^{-(N+1)\{k+v[\ell(Q)-j]+u\}}\r\}.
\end{align}
If $v[\ell(Q)-j]+u\leq-k$,
since $b^{-(N+1)\{k+v[\ell(Q)-j]+u\}}\geq1$,
obviously, \eqref{five13} still holds true.

Noticing that $a$ has the vanishing moments up to order $s$,
by Taylor's remainder theorem and $\varphi\in\cs(\rn)$,
we conclude that, for all $y\in x+B_{-k}$,
\begin{align}\label{five14}
&|a\ast\varphi_{-k}(y)|\\
&\hs=b^k\lf|\int_Q a(z)\lf\{\varphi(A^k(y-z))-
\sum_{|\gamma|\leq s}\frac
{\partial^\gamma\varphi(A^k(y-x_Q))}
{\gamma!}\lf[A^k(x_Q-z)\r]^\gamma\r\}\,dz\r|\noz\\
&\hs\ls b^k\int_Q|a(z)|\lf|A^k(x_Q-z)\r|
^{s+1}\sup_{\zeta\in Q}
\frac1{\max\lf\{1,\rho(A^k(y-\zeta))^{\widetilde{N}+1}\r\}}\,dz.\noz
\end{align}
Since $b^{\widetilde{N}}>b(\lambda_-)^{s+1}>b^{\frac{1+\epsilon}p}$,
combining \eqref{five12}, \eqref{five13} and \eqref{five14},
by the H\"{o}lder inequality and Lemma \ref{fivel1}(iv), we have
\begin{align}\label{five15}
{\rm II_2}
&=\lf[\sum_{-v[\ell(Q)-j]-u<k\leq-v\ell(Q)-u}+
\sum_{k\leq-v[\ell(Q)-j]-u}\r]
b^k\int_{x+B_{-k}}|a\ast\varphi_{-k}(y)|^2\,dy\\
&\ls\lf\{\sum_{-v[\ell(Q)-j]-u<k\leq-v\ell(Q)-u}
\lf[(\lambda_-)^{(s+1)[k+v\ell(Q)+u]}\r.\r.\noz\\
&\lf.\hs\times b^{-(\widetilde{N}+1)\{k+v[\ell(Q)-j]+u\}}
b^{k+\frac{v\ell(Q)}{r'}}\r]^2\noz\\
&\lf.\hs+\sum_{k\leq-v[\ell(Q)-j]-u}
\lf[(\lambda_-)^{(s+1)[k+v\ell(Q)+u]}
b^{k+\frac{v\ell(Q)}{r'}}\r]^2\r\}\|a\|
^2_{L^r(Q)}\noz\\
&\ls\lf\{\sum_{-v[\ell(Q)-j]-u<k
\leq-v\ell(Q)-u}b^{-\frac{2v\ell(Q)}r}
\lf[b(\lambda_-)^{s+1}\r]^{2vj}\r.\noz\\
&\hs+\lf.\sum_{k\leq-v[\ell(Q)-j]-u}
b^{-\frac{2v\ell(Q)}r}
\lf[b(\lambda_-)^{s+1}\r]^{2vj}\r\}\|a\|
^2_{L^r(Q)}\noz\\
&\ls(1-vj)b^{2vj\frac{1+\epsilon}p}
b^{-\frac{2v\ell(Q)}r}\|a\|^2_{L^r(Q)}.\noz
\end{align}
Combining the estimates of ${\rm II_1}$ (see \eqref{five11}) and
${\rm II_2}$ (see \eqref{five15}), we know that \eqref{five7} holds true.

\emph{Step 3.}
In this step, we use \eqref{five7}
to show $S(f)\in L^{p,q}(\rn)$ and
$\lf\|S(f)\r\|_{L^{p,q}(\rn)}\ls\|f\|_{H_A^{p,q}(\rn)}$.

To this end, it suffices only to consider
$N=N_{(p)}:=\lfloor(\frac1p-1)\frac{\ln b}{\ln \lambda_-}\rfloor+2$.
For all $k\in\mathbb{Z}$,
let $\mu_k:=(\sum_{i\in\mathbb{N}}|Q_i^k|)^{1/p}$.
Clearly, for $r\in(1,\fz]$, there exists $\delta\in(0,1)$ such that
$\max\{\frac1r,\frac1{1+\epsilon}\}<\delta<1$
and
$\delta p<1$, where $\epsilon$ is as in \eqref{five7}.
Notice that, for any fixed $k_0\in\zz$ and all $x\in\rn$,
$$S(f)(x)\leq S\lf(\sum_{k=-\infty}^{k_0-1}
\sum_{i\in\mathbb{N}}\lambda_i^ka_i^k\r)(x)+
\sum_{k={k_0}}^\infty \sum_{i\in\mathbb{N}}|
\lambda_i^k|S(a_i^k)(x)=:\psi_{k_0}(x)
+\eta_{k_0}(x).$$
To obtain the desired conclusion,
we consider two cases:
$q/p\in[1,\fz]$ and $q/p\in(0,1)$.

\emph{Case 1:}  $q/p\in[1,\fz]$. For this case,
if we prove that
\begin{align}\label{five16}
2^{k_0p}\lf[d_{\psi_{k_0}}(2^{k_0})\r]^\delta\ls
\sum_{k=-\infty}^{k_0-1}\lf[2^k\mu_k^\delta\r]^p\ \ \ \ {\rm and}\ \ \ \
2^{k_0\delta p}d_{\eta_{k_0}}(2^{k_0})\ls
\sum_{k=k_0}^\infty\lf[2^{k\delta}\mu_k\r]^p,
\end{align}
then, noticing that $\delta\in(0,q/p)$, by Lemma \ref{tl3}, the fact that
$|Q_i^k|\sim\frac{|\lambda_i^k|^p}{2^{kp}}$ and \eqref{five61},
we find that
\begin{align*}
\|S(f)\|_{L^{p,q}(\rn)}
\ls\lf\|\lf\{2^k\mu_k\r\}_{k\in\mathbb{Z}}\r\|_{\ell^q}
\ls\lf[\sum_{k\in\zz}\lf(\sum_{i\in\nn}|\lambda_i^k|^p\r)
^{\frac qp}\r]^{\frac 1q}\sim\|f\|_{H_A^{p,q}(\rn)}
\end{align*}
with the usual interpretation for $q=\fz$, which implies that
$S(f)\in L^{p,q}(\rn)$ and
$$\lf\|S(f)\r\|_{L^{p,q}(\rn)}\ls\|f\|_{H_A^{p,q}(\rn)}$$
as desired.

Now we show \eqref{five16}. To this end, we first
estimate $\psi_{k_0}$. Notice that $a_i^k$ is a $(p,r,s)$-atom,
$\supp a\subset Q_i^k$,
$\sum_{i\in\mathbb{N}}\chi_{Q_i^k}\ls1$
and $\lambda_i^k\sim2^k|Q_i^k|^{1/p}$. For $r\in(1,\fz)$,
by the H\"{o}lder inequality, we conclude that, for
$\sigma:=1-\frac p{r\delta}>0$ and all $x\in\rn$,
\begin{align*}
\psi_{k_0}(x)
&\leq\sum_{k=-\infty}^{k_0-1}S\lf(\sum_{i\in\mathbb{N}}
\lambda_i^ka_i^k\r)(x)\\
&\leq\lf(\sum_{k=-\infty}^{k_0-1}2^{k\sigma r'}\r)^{1/r'}
\lf\{\sum_{k=-\infty}^{k_0-1}2^{-k\sigma r}\lf[S\lf(
\sum_{i\in\mathbb{N}}\lambda_i^ka_i^k\r)(x)\r]^r\r\}^{1/r}\\
&=C_52^{k_0\sigma}
\lf\{\sum_{k=-\infty}^{k_0-1}2^{-k\sigma r}
\lf[S\lf(\sum_{i\in\mathbb{N}}\lambda_i^ka_i^k\r)(x)\r]^r\r\}^{1/r},
\end{align*}
where $C_5:=(\frac1{2^{\sigma r'}-1})^{1/r'}$,
which, together with \cite[Theorem 3.2]{blyz10},
further implies that
\begin{align}\label{five17}
&2^{k_0p}\lf[d_{\psi_{k_0}}(2^{k_0})\r]^\delta\\
&\hs\leq2^{k_0p}\lf|\lf\{x\in\rn:\ C_5^r
\sum_{k=-\infty}^{k_0-1}2^{-k\sigma r}
\lf[S\lf(\sum_{i\in\mathbb{N}}\lambda_i^ka_i^k\r)(x)\r]^
r>2^{k_0r(1-\sigma)}\r\}\r|^\delta\noz\\
&\hs\leq C_5^{\delta r}2^{k_0p}
2^{-k_0r\delta(1-\sigma)}\lf\{\int_{\rn}
\sum_{k=-\infty}^{k_0-1}2^{-k\sigma r}\lf[S\lf(
\sum_{i\in\mathbb{N}}\lambda_i^ka_i^k\r)(x)\r]^
r\,dx\r\}^\delta\noz\\
&\hs\ls\lf[\sum_{k=-\infty}^{k_0-1}2^{
-k\sigma r}\int_{\rn}\lf|\sum_{i\in\mathbb{N}}
\lambda_i^ka_i^k(x)\r|^r\,dx\r]^\delta\noz\\
&\hs\ls\lf[\sum_{k=-\infty}^{k_0-1}2^{-k\sigma r}
\sum_{i\in\mathbb{N}}|\lambda_i^k|^r
\int_{Q_i^k}|a_i^k(x)|^r\,dx\r]^\delta\noz\\
&\hs\ls\lf[\sum_{k=-\infty}^{k_0-1}2^{-k\sigma r}
\sum_{i\in\mathbb{N}}2^{kr}\lf|Q_i^k\r|^{\frac rp}
\lf|Q_i^k\r|^{(\frac 1r-\frac 1p)r}\r]^\delta\noz\\
&\hs\ls\sum_{k=-\infty}^{k_0-1}2^{kp}\lf(
\sum_{i\in\mathbb{N}}\lf|Q_i^k\r|\r)^\delta\sim\sum_
{k=-\infty}^{k_0-1}\lf[2^k\mu_k^\delta\r]^p,\noz
\end{align}
which is the desired estimate of
$\psi_{k_0}$ for $r\in(1,\fz)$ in \eqref{five16}.

For $r=\fz$, by \cite[Theorem 3.2]{blyz10} again,
we know that
\begin{align}\label{five41}
2^{k_0p}\lf[d_{\psi_{k_0}}(2^{k_0})\r]^\delta
&=2^{k_0p}\lf|\lf\{x\in\rn:\ \psi_{k_0}(x)
>2^{k_0}\r\}\r|^\delta\\
&\le2^{k_0(p-\delta\widetilde{r})}\lf\{\sum_{k=-\fz}^{k_0-1}\int_{\rn}
\lf[S\lf(\sum_{i\in\nn}\lambda_i^ka_i^k\r)(x)\r]^{\widetilde{r}}\,dx\r\}^\delta\noz\\
&\ls2^{k_0(p-\delta\widetilde{r})}
\lf\{\sum_{k=-\fz}^{k_0-1}\sum_{i\in\nn}\int_{x_i^k+B_{\ell_i^k}}
\lf|\lambda_i^ka_i^k\r|^{\widetilde{r}}(x)\,dx\r\}^\delta\noz\\
&\ls\sum_{k=-\infty}^{k_0-1}2^{kp}\lf(\sum_
{i\in\mathbb{N}}\lf|Q_i^k\r|\r)^\delta\sim\sum_
{k=-\infty}^{k_0-1}\lf[2^k\mu_k^\delta\r]^p,\noz
\end{align}
where $\widetilde{r}\in(1,\fz)$ satisfies that $\delta\widetilde{r}>p$, which,
combined with \eqref{five17}, implies the desired estimate of
$\psi_{k_0}$ in \eqref{five16}.

In order to estimate $\eta_{k_0}$, we claim that,
for all $i\in\mathbb{N}$ and $k\in\mathbb{Z}$,
\begin{align}\label{five18}
\int_\rn \lf[S(a_i^k)(x)\r]^{\delta p}\,dx\ls\lf|Q_i^k\r|^{1-\delta}.
\end{align}
Assume that \eqref{five18} holds true for the moment.
Then, by \eqref{five18}, we have
\begin{align}\label{five19}
\ \ \ 2^{k_0\delta p}d_{\eta_{k_0}}(2^{k_0})
&=2^{k_0\delta p}\lf|\lf\{x\in\rn:\ \lf[\sum_{k={k_0}}^\infty
\sum_{i\in\mathbb{N}}|\lambda_i^k|S(a_i^k)\r]^
{\delta p}(x)>2^{k_0\delta p}\r\}\r|\\
&\leq\int_\rn\lf[\sum_{k={k_0}}^\infty \sum_{i\in\mathbb{N}}|
\lambda_i^k|S(a_i^k)\r]^{\delta p}(x)\,dx\noz\\
&\leq\sum_{k={k_0}}^\infty\sum_{i\in\mathbb{N}}|\lambda_i^k|^
{\delta p}\int_\rn\lf[S(a_i^k)(x)\r]^{\delta p}\,dx\noz\\
&\ls\sum_{k={k_0}}^\infty\sum_{i\in\mathbb{N}}|\lambda_i^k|^
{\delta p}\lf|Q_i^k\r|^{1-\delta}\ls\sum_{k={k_0}}^\infty2^{k\delta p}
\sum_{i\in\mathbb{N}}\lf|Q_i^k\r|
\sim\sum_{k=k_0}^\infty\lf[2^{k\delta}\mu_k\r]^p\noz,
\end{align}
which is the desired estimate of
$\eta_{k_0}$ in \eqref{five16}.

Now we show \eqref{five18}. To this end, we write
\begin{align*}
\int_{\rn}\lf[S(a)(x)\r]^{\delta p}\,dx
&=\int_{x_{Q}+B_{v\ell(Q)+u+2\tau}}
\lf[S(a)(x)\r]^{\delta p}\,dx+
\int_{x_{Q}+(B_{v\ell(Q)+u+2\tau})^
\complement}\cdots\\
&=:{\rm I_1+I_2},
\end{align*}
where $a$ is a $(p,r,s)$-atom supported on the dyadic cube $Q$
and $v,\,u$ are as in Lemma \ref{fivel1}.
By the H\"{o}lder inequality, \cite[Theorem 3.2]{blyz10} again
and Lemma \ref{fivel1}(iv), we know that
\begin{align}\label{five20}
{\rm I_1}&\leq\lf\{\int_{x_{Q}+B_{v\ell(Q)+u+2\tau}}
\lf[S(a)(x)\r]^r\,dx\r\}^{\delta p/r}
\lf|B_{v\ell(Q)+u+2\tau}\r|^{1-\delta p/r}\\
&\ls\|a\|_{L^r(\rn)}^{\delta p}|Q|^{1-\delta p/r}
\ls|Q|^{\delta p(1/r-1/p)}|Q|^{1-\delta p/r}\sim|Q|^{1-\delta}.\noz
\end{align}
Noticing that $(1+\epsilon)\delta>1$, by \eqref{five7}
and Lemma \ref{fivel1}(iv) again, we find that
\begin{align*}
{\rm I_2}
&=\sum_{j=0}^\fz\int_{U_j}\lf[S(a)(x)\r]^{\delta p}\,dx\\
&\ls\sum_{j=0}^\fz(1-vj)^{\delta p/2}
b^{v\ell(Q)-vj}b^{(1+\epsilon)\delta vj}
b^{-\delta p\frac{v\ell(Q)}r}\|a\|^{\delta p}_{L^r(\rn)}\noz\\
&\ls\sum_{j=0}^\fz(1-vj)^{\delta p/2}
b^{v\ell(Q)-vj}b^{(1+\epsilon)\delta vj}b^{-\delta p
\frac{v\ell(Q)}r}b^{v\ell(Q)\delta p(\frac1r-\frac1p)}\noz\\
&\sim\sum_{j=0}^\fz b^{v\ell(Q)(1-\delta)}(1-vj)^
{\delta p/2}b^{[(1+\epsilon)\delta-1] vj}
\sim|Q|^{1-\delta},\noz
\end{align*}
where $U_j$ is as in \eqref{five21}, which, together with
\eqref{five20}, implies that \eqref{five18} holds true. This finishes
the proof of the case when $q/p\in[1,\fz]$.

\emph{Case 2:} $q/p\in(0,1)$. In this case, when $r\in(1,\fz)$,
similar to \eqref{five17}, we have
\begin{align}\label{five42}
2^{k_0p}\lf[d_{\psi_{k_0}}(2^{k_0})\r]^\delta
\ls\lf[\sum_{k=-\infty}^{k_0-1}2^{-k\sigma r}\sum_
{i\in\mathbb{N}}2^{kr}\lf|Q_i^k\r|^{\frac rp}\lf|Q_i^k\r|^
{(\frac 1r-\frac 1p)r}\r]^\delta\sim\lf(\sum_{k=-\infty}^
{k_0-1}2^{\frac{kp}\delta}\mu_k^p\r)^\delta.
\end{align}
By some calculations similar to \eqref{five41}, we easily know that
\eqref{five42} also holds true for $r=\fz$.
This further implies that
\begin{align}\label{five59}
\ \ \sum_{k_0\in\zz}2^{k_0q}\lf|\lf\{x\in\rn:\
\psi_{k_0}(x)>2^{k_0}\r\}\r|^{\frac qp}
&\ls\sum_{k_0\in\zz}2^{k_0(q-\frac q\delta)}
\sum_{k=-\infty}^{k_0-1}2^{\frac{kq}\delta}\mu_k^q\\
&\sim\sum_{k\in\zz}\sum_{k_0=k+1}^{\fz}
2^{k_0(q-\frac q\delta)}2^{\frac{kq}\delta}\mu_k^q
\ls\sum_{k\in\zz}2^{kq}\mu_k^q.\noz
\end{align}
On the other hand, similar to \eqref{five19}, we deduce that
\begin{align*}
2^{k_0\delta p}d_{\eta_{k_0}}(2^{k_0})
\ls\sum_{k=k_0}^\infty\lf[2^{k\delta}\mu_k\r]^p,
\end{align*}
which implies that
\begin{align*}
&2^{k_0\delta p}\lf|\lf\{x\in\rn:\
\eta_{k_0}(x)>2^{k_0}\r\}\r|\\
&\hs\ls\sum_{k=k_0}^\fz2^{-k\widetilde{\delta}p}
\lf[2^{k(1-\widetilde{\delta})}\mu_k\r]^p
\ls2^{-k_0\widetilde{\delta}p}\lf\{\sum_{k=k_0}^\fz
\lf[2^{k(1-\widetilde{\delta})}\mu_k\r]^q\r\}^{\frac pq},
\end{align*}
where $\widetilde{\delta}:=\frac{1-\delta}2$. Therefore,
\begin{align}\label{five60}
\sum_{k_0\in\zz}2^{k_0q}\lf|\lf\{x\in\rn:\
\eta_{k_0}(x)>2^{k_0}\r\}\r|^{\frac qp}
&\ls\sum_{k_0\in\zz}2^{k_0\widetilde{\delta}q}\sum_
{k=k_0}^\fz\lf[2^{k(1-\widetilde{\delta})}\mu_k\r]^q\\
&\sim\sum_{k\in\zz}\lf[2^{k(1-\widetilde{\delta})}\mu_k\r]^q
\sum_{k_0=-\fz}^k2^{k_0\widetilde{\delta}q}
\ls\sum_{k\in\zz}2^{kq}\mu_k^q.\noz
\end{align}
Notice that
$\mu_k:=(\sum_{i\in\mathbb{N}}|Q_i^k|)^{1/p}$
and $\lambda_i^k\sim2^k|Q_i^k|^{1/p}$.
Combining \eqref{se6}, \eqref{five59}, \eqref{five60}
and \eqref{five61}, we further conclude that
\begin{align*}
\|S(f)\|_{L^{p,q}(\rn)}^q
&\sim\sum_{k_0\in\zz}2^{k_0q}
\lf|\lf\{x\in\rn:\ S(f)(x)>2^{k_0}\r\}\r|^{\frac qp}\\
&\ls\sum_{k_0\in\zz}2^{k_0q}\lf|\lf\{x\in\rn:\
\psi_{k_0}(x)>2^{k_0}\r\}\r|^{\frac qp}\\
&\hs+\sum_{k_0\in\zz}2^{k_0q}\lf|\lf\{x\in\rn:\
\eta_{k_0}(x)>2^{k_0}\r\}\r|^{\frac qp}\noz\\
&\ls\sum_{k\in\zz}2^{kq}\mu_k^q\sim\sum_{k\in\zz}
\lf[\sum_{i\in\nn}|\lambda_i^k|^p\r]^{\frac qp}
\sim\|f\|_{H_A^{p,q}(\rn)}^q,
\end{align*}
which implies that $S(f)\in L^{p,q}(\rn)$ and
$\lf\|S(f)\r\|_{L^{p,q}(\rn)}\ls\|f\|_{H_A^{p,q}(\rn)}$.
This finishes the proof of Case 2 and hence the proof of
the necessity of Theorem \ref{fivet1}.

Now we show the sufficiency of Theorem \ref{fivet1},
namely, to show that, if $f\in\cs'_0(\rn)$ and $S(f)\in L^{p,q}(\rn)$,
then $f\in H_A^{p,q}(\rn)$ and
$$\|f\|_{H_A^{p,q}(\rn)}\ls\|S(f)\|_{L^{p,q}(\rn)}.$$
We prove this by six steps.

\emph{Step (i).} For each $k\in\mathbb{Z}$, let
$\Omega_k:=\{x\in\rn:\ S(f)(x)>2^k\}$ and
$$\mathcal{Q}_k:=\lf\{Q\in\mathcal{Q}:
\ |Q\cap\Omega_k|>\frac{|Q|}2\ \ {\rm and}\
\ |Q\cap\Omega_{k+1}|\leq\frac{|Q|}2\r\}.$$
Obviously, for any $Q\in\mathcal{Q}$,
there exists a unique $k\in\mathbb{Z}$
such that $Q\in\mathcal{Q}_k$.
We denote the set of all \emph{maximal dyadic cubes}
in $\mathcal{Q}_k$ by $\{Q_i^k\}_i$,
namely, there exist no $Q\in\mathcal{Q}_k$
such that $Q_i^k\subsetneqq Q$ for any $i$.

For any $Q\in\mathcal{Q}$, let
$$\widehat{Q}:=\lf\{(y,t)\in\rn\times\mathbb{R}:\
y\in Q,\,t\sim v\ell(Q)+u\r\},$$
here and hereafter, $t\sim v\ell(Q)+u$ always means
\begin{align}\label{five23}
v\ell(Q)+u+\tau\leq t<v[\ell(Q)-1]+u+\tau,
\end{align}
where $u,\,v$ are as in Lemma \ref{fivel1} and $\ell(Q)$ is the level of $Q$.
Observe that, in the above inequality, $v$ is negative. Clearly,
$\{\widehat{Q}\}_{Q\in\mathcal{Q}}$ are mutually disjoint and
\begin{align}\label{five24}
\rn\times\mathbb{R}=\bigcup_{k\in\mathbb{Z}}\bigcup_i B_{k,\,i},
\end{align}
where, for any $k\in\mathbb{Z}$ and $i$, let
$B_{k,\,i}:=\bigcup_{Q\subset Q_i^k,\,Q\in\mathcal{Q}_k}\widehat{Q}$.
It is easy to see that $\{B_{k,i}\}_{k\in\zz,\,i}$
are mutually disjoint by Lemma \ref{fivel1}(ii).

Let $\psi$ and $\theta$ be as in Lemma \ref{fivel2}. Then
$\theta$ has the vanishing moments up to order $s$ with
$s\geq\lfloor(1/p-1)\ln b/\ln\lambda_-\rfloor$. By
Lemma \ref{fivel2}, the properties of the tempered distributions
(see \cite[Theorem 2.3.20]{lg08} or \cite[Theorem 3.13]{sw71}) and \eqref{five24},
we find that, for all $f\in\cs'_0(\rn)$ with
$S(f)\in L^{p,q}(\rn)$, and $x\in\rn$,
\begin{align*}
f(x)
&=\sum_{k\in\mathbb{Z}}f\ast\psi_k\ast\theta_k(x)
=\int_{\rn\times\mathbb{R}}
f\ast\psi_t(y)\ast\theta_t(x-y)\,dy\,dm(t)\\
&=\sum_{k\in\mathbb{Z}}\sum_i\int_{B_{k,\,i}}
f\ast\psi_t(y)\ast\theta_t(x-y)\,dy\,dm(t)
=:\sum_{k\in\mathbb{Z}}\sum_i h_i^k(x)
\end{align*}
in $\cs'(\rn)$, where, for each $k\in\mathbb{Z},\,i$ and $x\in\rn$,
\begin{align}\label{five25}
h_i^k(x)
:=&\int_{B_{k,\,i}}f\ast\psi_t(y)\ast\theta_t(x-y)\,dy\,dm(t)\\
=&\sum_{Q\subset Q_i^k,\,Q\in\mathcal{Q}_k}
\int_{\widehat{Q}}f\ast\psi_t(y)\ast\theta_t(x-y)\,dy\,dm(t)
=:\sum_{Q\subset Q_i^k,\,Q\in\mathcal{Q}_k}e_{Q}(x)\noz
\end{align}
with convergence in $\cs'(\rn)$, and $m(t)$ is
the \emph{counting measure} on $\mathbb{R}$.

\emph{Step (ii).} In this step, we prove that,
for all $x\in\rn$,
\begin{align}\label{five27}
\lf[S\lf(\sum_{Q\in\mathcal{R}}e_Q\r)(x)\r]^2\ls
\sum_{Q\in\mathcal{R}}\lf[M_{{\rm HL}}(c_Q\chi_Q)(x)\r]^2,
\end{align}
where $\mathcal{R}$ is any set of dyadic cubes in
$\rn$, $e_Q$ is as in \eqref{five25}
and, for each $Q\in\mathcal{R}$,
$$c_Q:=\lf[\int_{\widehat{Q}}|\psi_t\ast f(y)|^2
\,dy\frac{dm(t)}{b^{t}}\r]^{1/2}.$$
Indeed,
we can prove \eqref{five27} via borrowing some ideas from
the proof of \cite[Lemma 4.7]{blyz10}. For each $Q\in\mathcal{Q}$,
let $\widetilde{Q}:=c_Q+B_{v[\ell(Q)-1]+u+2\tau}$,
where $c_Q$ is as in \eqref{five27}. By repeating the proof of
\cite[Lemma 4.7, pp.\,412-413]{blyz10}, we find that,
for all $x\in\rn$,
\begin{align*}
\lf[S\lf(\sum_{Q\in\mathcal{R}}e_Q\r)(x)\r]^2
\ls\sum_{Q\in\mathcal{R}}(c_Q)^2
[M_{{\rm HL}}(\chi_Q)(x)]^2\lf\{\sum_{R\in\mathcal{Q},\,
\widetilde{Q}\cap\widetilde{R}\neq\emptyset}b^
{(s+1)v|\ell(Q)-\ell(R)|\frac{\ln(\lambda_-)}{\ln b}}\r\}.
\end{align*}
By this and the following estimate from \cite[p.\,295]{lfy14}:
$$\sum_{R\in\mathcal{Q},\,\widetilde{Q}\cap
\widetilde{R}\neq\emptyset}b^{(s+1)v
|\ell(Q)-\ell(R)|\frac{\ln(\lambda_-)}{\ln b}}\ls1$$
(see also \cite[(4.18)]{blyz10} for the case of two parameters), we obtain \eqref{five27}.

Next we show that, for each $k\in\zz$ and $i$,
$h_i^k$ is a multiple of a $(p,r,s)$-atom. This is completed
by Step (iii) through (v) below.

\emph{Step (iii).}
For all $x\in\supp h_i^k$, by \eqref{five25}, $h_i^k(x)\neq0$
implies that there exists $Q\subset Q_i^k$ and
$Q\in\mathcal{Q}_k$ such that $e_{Q}(x)\neq0$.
Then there exists $(y,t)\in\widehat{Q}$ such that
$A^{-t}(x-y)\in B_0$. By this, Lemma \ref{fivel1}(iv),
\eqref{five23} and \eqref{se1}, we have
$$x\in y+B_t\subset x_Q+B_{v\ell(Q)+u}+B_{v[\ell(Q)-1]
+u+\tau}\subset x_Q+B_{v[\ell(Q)-1]+u+2\tau}.$$
Therefore,
$$\supp e_Q\subset x_Q+B_{v[\ell(Q)-1]+u+2\tau}.$$
By this,
$h_i^k=\sum_{Q\subset Q_i^k,\,Q\in\mathcal{Q}_k}e_{Q}$,
(ii) and (iv) of Lemma \ref{fivel1} and \eqref{se1},
we further conclude that
\begin{align}\label{five26}
\supp h_i^k
&\subset\bigcup_{Q\subset Q_i^k,\,Q\in\mathcal
{Q}_k}x_Q+B_{v[\ell(Q)-1]+u+2\tau}\\
&\subset x_{Q_i^k}+B_{v\ell(Q_i^k)+u}+
B_{v[\ell(Q_i^k)-1]+u+2\tau}\noz\\
&\subset x_{Q_i^k}+B_{v[\ell(Q_i^k)-1]+
u+3\tau}=:B_i^k.\noz
\end{align}

\emph{Step (iv).}
Notice that, for each $Q\in\mathcal{Q}_k$ and
$x\in Q$, by Lemma \ref{fivel1}(iv), we know that
$$M_{{\rm HL}}(\chi_{\Omega_k})(x)\geq\frac1{b^{v\ell(Q)+u}}
\int_{x_Q+B_{v\ell(Q)+u}}\chi_{\Omega_k}(z)\,dz>b^{-2u}
\frac{|\Omega_k\cap Q|}{|Q|}>\frac12b^{-2u},$$
which further implies that
\begin{align}\label{five28}
\bigcup_{Q\subset Q_i^k,\,Q\in\mathcal{Q}_k}Q
\subset\widehat{\Omega}_k:=\lf\{x\in\rn:\
M_{{\rm HL}}(\chi_{\Omega_k})(x)>\frac12b^{-2u}\r\}.
\end{align}
Moreover, for all $Q\in\mathcal{Q}_k$ and $x\in Q$,
by Lemma \ref{fivel1}(iv) and
$Q\subset\widehat{\Omega}_k$, we obtain
$$M_{{\rm HL}}\lf(\chi_{Q\cap(\widehat{\Omega}_k
\setminus\Omega_{k+1})}\r)(x)
\gs\frac1{|Q|}\int_{Q}\chi_{\widehat{\Omega}_k
\setminus\Omega_{k+1}}(z)\,dz
\gs\frac{|Q|-|Q|/2}{|Q|}\gs\frac{\chi_Q(x)}2.$$
By this, \cite[Theorem 3.2]{blyz10}, \eqref{five27}
and \cite[Theorem 2.5]{bh06},
we find that, for $r\in(1,\fz)$,
\begin{align}\label{five29}
\lf\|\sum_{Q\subset Q_i^k,\,Q\in\mathcal{Q}_k}
e_{Q}\r\|_{L^r(\rn)}
&\ls\lf\|S\lf(\sum_{Q\subset Q_i^k,\,Q\in\mathcal{Q}_k}
e_{Q}\r)\r\|_{L^r(\rn)}\\
&\ls\lf\|\lf\{\sum_{Q\subset Q_i^k,\,Q\in\mathcal{Q}_k}
\lf[M_{{\rm HL}}(c_Q\chi_Q)\r]^2\r\}^{1/2}\r\|_{L^r(\rn)}\noz\\
&\ls\lf\|\lf[\sum_{Q\subset Q_i^k,\,Q\in\mathcal{Q}_k}
(c_Q)^2\chi_Q\r]^{1/2}\r\|_{L^r(\rn)}\noz\\
&\ls\lf\|\lf\{\sum_{Q\subset Q_i^k,\,Q\in\mathcal{Q}_k}
\lf[M_{{\rm HL}}\lf(c_Q\chi_{Q\cap(\widehat{\Omega}_k
\setminus\Omega_{k+1})}\r)\r]^2
\r\}^{1/2}\r\|_{L^r(\rn)}\noz\\
&\ls\lf\|\lf[\sum_{Q\subset Q_i^k,\,Q\in\mathcal{Q}_k}
\lf(c_Q\r)^2\chi_{Q\cap(\widehat{\Omega}_k
\setminus\Omega_{k+1})}\r]^{1/2}\r\|_{L^r(\rn)}.\noz
\end{align}

On the other hand, for any $Q\in\mathcal{Q}_k,\,x\in Q$
and $(y,t)\in\widehat{Q}$, by Lemma
\ref{fivel1}(iv), \eqref{se1} and \eqref{five23},
we have
$$x-y\in B_{v\ell(Q)+u}+B_{v\ell(Q)+u}
\subset B_{v\ell(Q)+u+\tau}\subset B_t,$$
which, combined with the disjointness of
$\{\widehat{Q}\}_{Q\subset Q_i^k}$,
further implies that
\begin{align}\label{five30}
&\sum_{Q\subset Q_i^k,\,Q\in\mathcal{Q}_k}
\lf(c_Q\r)^2\chi_{Q\cap(\widehat{\Omega}_k
\setminus\Omega_{k+1})}(x)\\
&\hs=\sum_{Q\subset Q_i^k,\,Q\in\mathcal{Q}_k}
\int_{\widehat{Q}}
|\psi_t\ast f(y)|^2\,dy\frac{dm(t)}{b^{t}}
\chi_{Q\cap(\widehat{\Omega}_k
\setminus\Omega_{k+1})}(x)\noz\\
&\hs\ls\lf[S(f)(x)\r]^2\chi_{Q_i^k\cap(\widehat
{\Omega}_k\setminus\Omega_{k+1})}(x).\noz
\end{align}
By the definition of $\widehat{\Omega}_k$
(see \eqref{five28}), we know that, for all $r\in(1,\fz)$,
\begin{align*}
|\widehat{\Omega}_k|\leq(2b^{2u})^r\int_{\rn}
\lf[M_{{\rm HL}}(\chi_{\Omega_k})(x)\r]^r\,dx
\ls|\Omega_k|,
\end{align*}
which, together with \eqref{five30}, implies that
\begin{align}\label{five33}
&\lf\|\lf\{\sum_{Q
\subset Q_i^k,\,Q\in\mathcal{Q}_k}
\lf(c_Q\r)^2\chi_{Q\cap(\widehat{\Omega}_k
\setminus\Omega_{k+1})}\r\}^
{\frac12}\r\|^r_{L^r(\rn)}\\
&\hs\leq\int_{\rn}
\lf[\chi_{Q_i^k\cap(\widehat{\Omega}_k
\setminus\Omega_{k+1})}(x)
\int_{\bigcup_{Q\subset Q_i^k,\,
Q\in\mathcal{Q}_k}\widehat{Q}}
|\psi_t\ast f(y)|^2\,dy\frac{dm(t)}
{b^{t}}\r]^{r/2}\,dx\noz\\
&\hs\ls 2^{kr}\lf|\widehat{\Omega}_k\r|
\ls2^{kr}|\Omega_k|<\fz.\noz
\end{align}
For any $N\in\mathbb{N}$, let
$\mathcal{Q}_{k,\,N}:=
\{Q\in\mathcal{Q}_k:\ |\ell(Q)|>N\}$.
Notice that, if we replace
$\sum_{Q\subset Q_i^k,\,Q\in\mathcal{Q}_k}e_{Q}$
by $\sum_{Q\subset Q_i^k,\,Q\in\mathcal{Q}_{k,N}}e_Q$
in \eqref{five29}, then
\begin{align*}
&\lf\|\sum_{Q\subset Q_i^k,\,Q\in
\mathcal{Q}_{k,\,N}}e_{Q}\r\|_{L^r(\rn)}^r\\
&\hs\ls\lf\|\lf[\sum_{Q\subset Q_i^k,\,Q\in
\mathcal{Q}_{k,\,N}}\lf(c_Q\r)^2\chi_{Q\cap
(\widehat{\Omega}_k\setminus\Omega_{k+1})}
\r]^{1/2}\r\|_{L^r(\rn)}^r\noz\\
&\hs\ls\int_{\rn}\chi_{Q_i^k\cap
(\widehat{\Omega}_k\setminus\Omega_{k+1})}(x)
\lf[\int_{\bigcup_{Q\subset Q_i^k,\,Q\in
\mathcal{Q}_{k,\,N}}\widehat{Q}}|\psi_t\ast f(y)|^2\,dy
\frac{dm(t)}{b^{t}}\r]^{r/2}\,dx.\noz
\end{align*}
Thus, by \eqref{five33} and the
Lebesgue dominated convergence theorem, we conclude that
$$\lf\|\sum_{Q\subset Q_i^k,\,Q\in
\mathcal{Q}_{k,\,N}}e_{Q}\r\|_{L^r(\rn)}\rightarrow0$$
as $N\rightarrow\fz$, which implies that
$h_i^k=\sum_{Q\subset Q_i^k,\,Q\in\mathcal{Q}_k}e_{Q}$
in $L^r(\rn)$. By this, combined with \eqref{five29},
\eqref{five30}, the definition of $B_i^k$ (see \eqref{five26})
and Lemma \ref{fivel1}(iv), we further find that
\begin{align}\label{five31}
\|h_i^k\|_{L^r(\rn)}&\ls\lf\{\int_{\rn}\lf[S(f)(x)\r]^r
\chi_{Q_i^k\cap(\widehat{\Omega}_k
\setminus\Omega_{k+1})}(x)\,dx\r\}^{1/r}\\
&\ls2^k|Q_i^k|^{1/r}
\le C_62^k|B_i^k|^{1/r},\noz
\end{align}
where $C_6$ is a positive constant
independent of $f$.

\emph{Step (v).}
Recall that $\theta$ has the vanishing moments up to
$s\geq\lfloor(1/p-1)\ln b/\ln\lambda_-\rfloor$
and so is $e_Q$. For all
$k\in\mathbb{Z},\,i$, $\gamma\in\zz_+^n$ with
$|\gamma|\leq s$ and $x\in\rn$,
let $g(x):=x^\gamma\chi_{B_i^k}(x)$ with $B_i^k$ being as in
\eqref{five26}, and $r'\in(1,\fz)$ such that $1/r+1/r'=1$.
Clearly, $g\in L^{r'}(\rn)$.
Therefore, by the fact that
$(L^{r'}(\rn))^\ast=L^r(\rn)$, \eqref{five26} and
$$\supp e_Q\subset x_Q+
B_{v[\ell(Q)-1]+u+2\tau}\subset B_i^k,$$
we further have
\begin{align}\label{five34}
\int_{\rn}h_i^k(x)x^\gamma\,dx
=\langle h_i^k,g\rangle
&=\sum_{Q\subset Q_i^k,\,Q\in
\mathcal{Q}_k}\langle e_{Q},g\rangle\\
&=\sum_{Q\subset Q_i^k,\,Q\in\mathcal{Q}_k}
\int_{\rn}e_{Q}(x)x^\gamma\,dx=0.\noz
\end{align}
Namely, $h_i^k$ has the vanishing moments up to $s$,
which, together with \eqref{five26} and \eqref{five31},
implies that $h_i^k$ is a $(p,r,s)$-atom modulus a constant
supported on $B_i^k$.

\emph{Step (vi).}
For all $k\in\mathbb{Z}$ and $i$,
let $\lambda_i^k:=C_62^k|B_i^k|^{1/p}$ and
$a_i^k:=(\lambda_i^k)^{-1}h_i^k$, where
$C_6$ is a positive constant as in \eqref{five31}.
Then we have
$$f=\sum_{k\in\mathbb{Z}}\sum_i h_i^k=
\sum_{k\in\mathbb{Z}}\sum_i \lambda_i^k a_i^k\qquad {\rm in}\quad \cs'(\rn).$$
By \eqref{five26} and \eqref{five34},
we find that $\supp a_i^k\subset B_i^k$ and
$a_i^k$ also has the vanishing moments up to $s$.
By \eqref{five31} and Lemma \ref{fivel1}(iv),
we conclude that
$\|a_i^k\|_{L^r(\rn)}\leq|B_i^k|^{1/r-1/p}$.
Thus, $a_i^k$ is a $(p,r,s)$-atom
for all $k\in\mathbb{Z}$ and $i$.

By the mutual disjointness of $\{Q_i^k\}_{k\in\mathbb{Z},\,i}$,
Lemma \ref{fivel1}(iv) again and \eqref{se6},
we know that
\begin{align*}
\sum_{k\in\mathbb{Z}}
\lf(\sum_i |\lambda_i^k|^p\r)^{\frac qp}
&\sim\sum_{k\in\mathbb{Z}}
\lf(\sum_i 2^{kp}\lf|B_i^k\r|\r)^{\frac qp}\\
&\sim \sum_{k\in\mathbb{Z}}2^{kq}
\lf(\sum_i\lf|Q_i^k\r|\r)^{\frac qp}\ls
\sum_{k\in\mathbb{Z}}2^{kq}|\Omega_k|^
{\frac qp}\sim\|S(f)\|^q_{L^{p,q}(\rn)},
\end{align*}
which implies that $f\in H_A^{p,q}(\rn)$ and
$\|f\|_{H_A^{p,q}(\rn)}\ls\|S(f)\|_{L^{p,q}(\rn)}$.
This finishes the proof of Theorem \ref{fivet1}.
\end{proof}

\section{Proofs of Theorems \ref{fivet2} and \ref{fivet3}\label{s3}}

\hskip\parindent
In this section, via first establishing
an anisotropic Fefferman-Stein vector-valued inequality in
$L^{p,q}(\rn)$ and the Calder\'{o}n reproducing formula,
we give out the proofs of Theorems \ref{fivet2} and \ref{fivet3}.

We begin with recalling some notation and
establishing several technical lemmas.
Suppose that $A$ is a dilation on $\rn$.
For each $j\in\mathbb{Z}$ and $k\in\mathbb{Z}^n$,
define $Q_{j,\,k}:=A^{-j}([0,1)^n+k)$,
$\mathcal{Q}_j:=\{Q_{j,\,k}:\ k\in\mathbb{Z}^n\}$
and
$\widehat{\mathcal{Q}}:=
\bigcup_{j\in\mathbb{Z}}\mathcal{Q}_j$.
Recall that $Q_{j,\,k}$ with $j\in\mathbb{Z}$
and
$k\in\mathbb{Z}^n$ is called a \emph{dilated cube}
(see, for example, \cite[p.\,1475]{bh06}). Clearly,
for any $k_1,\,k_2\in\mathbb{Z}^n$ with $k_1\neq k_2$,
$|Q_{j,\,k_1}\cap Q_{j,\,k_2}|=0$.

Throughout this article, for each $r\in(0,\fz)$ and $x\in\rn$,
let
$$B_\rho(x,\,r):=\{y\in\rn:\ \rho(x-y)<r\}.$$
For each dilated cube $Q_{j,\,k}$,
denote its \emph{center} by $c_{Q_{j,\,k}}$
and its
\emph{lower-left corner} $A^{-j}k$ by $x_{Q_{j,\,k}}$.
Via \cite[Lemma 2.9(a)]{mb07}, we know that
there exists a positive integer $j_0:=j_{(A,\,n)}$,
only depending on $A$ and $n$, such that,
for each $x\in Q_{j,\,k}$,
\begin{align}\label{five35}
B_\rho(c_{Q_{j,\,k}}, b^{-j_0-j})
\subset Q_{j,\,k}\subset B_\rho(x, b^{j_0-j}).
\end{align}
For any $\varphi\in\cs(\rn)$ and $x\in\rn$, define
\begin{align}\label{five36}
\varphi_{Q}(x):=|\det A|^{j/2}\varphi(A^jx-k)
=|Q|^{1/2}\varphi_{-j}(x-x_Q),
\end{align}
where $Q:=Q_{j,\,k}\in\widehat{\mathcal{Q}}$
and $x_Q:=x_{Q_{j,\,k}}$.

Observe that $(\rn,\,\rho,\,dx)$ is a space
of homogeneous type in the sense of
Coifman and Weiss \cite{cw71,cw77}. From this and
\cite[Theorem 1.2]{gly09}, we deduce the following lemma,
which is an anisotropic version of \cite[Theorem 1]{fs71}.

\begin{lemma}\label{fivel3}
Let $r\in(1,\fz]$ and $M_{{\rm HL}}$
be the Hardy-Littlewood maximal function
defined by \eqref{te58}.
\begin{enumerate}
\item[{\rm(i)}]
If $p\in(1,\fz)$, then there exists a positive constant $C_7$ such that,
for all sequences $\{f_k\}_k$ of measurable functions,
\begin{align*}
\lf\|\lf\{\sum_{k}\lf[M_{{\rm HL}}(f_k)\r]^r\r\}^
{1/r}\r\|_{L^p(\rn)}\leq C_7\lf\|\lf[\sum_{k}
|f_k|^r\r]^{1/r}\r\|_{L^p(\rn)};
\end{align*}
\item[{\rm(ii)}]
It holds true that there exists a positive constant $C_8$ such that,
for all sequences $\{f_k\}_k$ of measurable functions,
\begin{align*}
\lf\|\lf\{\sum_{k}\lf[M_{{\rm HL}}(f_k)\r]^r\r\}^
{1/r}\r\|_{L^{1,\fz}(\rn)}\leq C_8\lf\|\lf[\sum_{k}
|f_k|^r\r]^{1/r}\r\|_{L^1(\rn)}.
\end{align*}
\end{enumerate}
\end{lemma}

The following lemma is just
\cite[Lemma 3.3]{lfy14}.
\begin{lemma}\label{fivel4}
Suppose that
$\ell,\,M\in\mathbb{N}$ and $\varphi,\,\psi\in\cs(\rn)$
which have the vanishing moments up to order $\ell$.
Then there exists a positive constant
$C_{(\ell,\,M)}$, depending on $\ell$ and $M$,
such that, for any
$i,\,j\in\mathbb{Z}$ with $i\geq j$ and $x\in\rn$,
$$\lf|\varphi_{-i}\ast\psi_{-j}(x)\r|\leq
C_{(\ell,\,M)}b^{j-(i-j)(\ell+1)\zeta_-}
\lf[1+\rho(A^jx)\r]^{-M},$$
where $\zeta_-:=\ln(\lambda_-)/\ln b$.
\end{lemma}

The following discrete Calder\'{o}n reproducing formula is just
\cite[Lemma 3.2]{lfy14},
which is an anisotropic version of
\cite[Theorem 6.16]{fjw91}.

\begin{lemma}\label{fivel5}
Suppose that $\Psi,\,\Phi\in\cs(\rn)$ satisfy that
$\supp \widehat{\Psi},\,\supp \widehat{\Phi}
\subset [-1,1]^n\setminus\{\vec0_n\}$
and, for all $\xi\in\rn\setminus\{\vec0_n\}$,
$$\sum_{j\in\mathbb{Z}}\overline
{\widehat{\Phi}((A^*)^j\xi)}\widehat{\Psi}((A^*)^j\xi)=1,$$
where $A^*$ denotes the transpose of $A$.
Then, for any $f\in\cs'_0(\rn)$,
$$f(\cdot)=\sum_{Q\in\widehat{\mathcal{Q}}}
\langle f,\Phi_Q\rangle\Psi_Q(\cdot)
=\sum_{j\in\mathbb{Z}}\sum_{Q\in\mathcal{Q}_j}b^{j}
f\ast\widetilde{\Phi}_{j}(x_Q)\Psi_{j}(\cdot-x_Q)\hs
holds\ true\ in\hs \cs'(\rn),$$
where
$\widetilde{\Phi}(\cdot):=\overline{\Phi(-\cdot)}$,
and $\Phi_Q,\,\Psi_Q$ are defined as in \eqref{five36}.
\end{lemma}

The dyadic maximal function $M_d(f)$ is defined
by setting, for all $f\in L_{{\rm loc}}^1(\rn)$ and $x\in\rn$,
\begin{align}\label{five70}
M_d(f)(x):=\sup_{k\in\zz}E_k(f)(x),
\end{align}
where, for any $k\in\zz$,
$$E_k(f)(x):=\sum_{Q\in\mathcal{Q}_k}\lf[\frac1{|Q|}
\int_{Q}|f(y)|\,dy\r]\chi_Q(x)$$
and $\mathcal{Q}_k:=\{Q_\alpha^k:\ \alpha\in I_k\}$
denotes the set of dyadic cubes from Lemma
\ref{fivel1}. Moreover, by \cite[Proposition A.4(ii)]{blyz10},
we know that $f\le M_d(f)$ almost everywhere.

By a slight modification on the proof of
\cite[Proposition A.5]{blyz10}, we easily find that
the conclusions of \cite[Proposition A.5]{blyz10} also hold
true for all $f\in L^{p,q}(\rn)$ with $p\in(1,\fz)$ and $q\in(0,\fz]$,
the details being omitted. This provides the Calder\'{o}n-Zygmund
decomposition in the present setting, which is stated as follows.

\begin{lemma}\label{fivel12}
Let $p\in(1,\fz)$, $q\in(0,\fz]$ and $f\in L^{p,q}(\rn)$.
Then, for all $\lambda\in(0,\fz)$, there exists a sequence
$\{Q_j\}_j\subset\mathcal{Q}$ of mutually disjoint dyadic cubes such that
\begin{enumerate}
\item[{\rm (i)}] $\bigcup_{j}Q_j=\{x\in\rn:\ M_d(f)(x)>\lambda\}$;
\item[{\rm (ii)}] $|f(x)|\le\lambda$ for almost every $x\notin\bigcup_{j}Q_j$;
\item[{\rm (iii)}] there exists a constant $C\in(1,\fz)$, independent of
$f$ and $\lambda$, such that
$$\lambda<\frac1{|Q_j|}\int_{Q_j}|f(x)|\,dx\le C\lambda;$$
\item[{\rm (iv)}] for any $Q\in\{Q_j\}_j$, there exists a unique
$\widetilde{Q}\in\mathcal{Q}$ such that
$$Q\subset\widetilde{Q},\ \ \ell(Q)=\ell(\widetilde{Q})-1\ \ and\ \
\frac1{|\widetilde{Q}|}\int_{\widetilde{Q}}|f(x)|\,dx<\lambda,$$
where $\mathcal{Q}$ is as in Lemma \ref{fivel1} and $\ell(Q)$
is the level of $Q$.
\end{enumerate}
\end{lemma}

Motivated by \cite[Theorem 1]{ycp05}, we obtain the following anisotropic
Fefferman-Stein vector-valued inequality in $L^{p,q}(\rn)$,
which plays a key role in the proof of Theorem \ref{fivet2}.

\begin{lemma}\label{fivel6}
Let $r\in(1,\fz]$.
\begin{enumerate}
\item[{\rm(i)}]
If $p\in(1,\fz)$ and $q\in(0,\fz]$, then
there exists a positive constant $C_9$ such that,
for all sequences
$\{f_j\}_j$ of measurable functions,
\begin{align}\label{five22}
\lf\|\lf\{\sum_{j}\lf[M_{{\rm HL}}(f_j)\r]^r\r\}^
{\frac1r}\r\|_{L^{p,q}(\rn)}
\leq C_9\lf\|\lf[\sum_{j}\lf|f_j\r|^r\r]^
{\frac1r}\r\|_{L^{p,q}(\rn)};
\end{align}
\item[{\rm(ii)}]
If $p\in(0,\fz),\,q\in(0,\fz]$
and $s\in(0,\min\{r,p\})$, then
there exists a positive constant $C_{10}$
such that, for all sequences
$\{f_j\}_j$ of measurable functions,
\begin{align}\label{five62}
\lf\|\lf\{\sum_{j}\lf[M_{{\rm HL}}(f_j)\r]^
{\frac rs}\r\}^{\frac1r}\r\|_{L^{p,q}(\rn)}
\leq C_{10}\lf\|\lf[\sum_{j}\lf|f_j\r|^{\frac rs}\r]^
{\frac1r}\r\|_{L^{p,q}(\rn)}.
\end{align}
\end{enumerate}
\end{lemma}

\begin{proof}
Assume that the right-hand sides of \eqref{five22} and \eqref{five62}
are finite. We first prove (i). To this end,
let $f:=[\sum_{j}|f_j|^r]^{\frac 1r}$ and
$\Omega_k:=\{x\in\rn:\ M_d(f)(x)>2^k\}$ for any $k\in\zz$,
where $M_d(f)$ is as in \eqref{five70}.
By (i) and (iii) of Lemma \ref{fivel12}, we obtain a sequence
$\{Q_i^k\}_{i}$ of dyadic cubes
satisfying that
$\Omega_k=\bigcup_iQ_i^k$,
\begin{align}\label{five4}
Q_i^k\cap Q_j^k
=\emptyset\ \ \ {\rm for\ all}\  i,\ j\ {\rm with}\ i\neq j
\end{align}
and, for all $i$,
\begin{align}\label{five5}
\frac1{|Q_i^k|}
\int_{Q_i^k}f(x)\,dx\ls2^k.
\end{align}
For any $j$, let $f_j^{(1)}:=f_j\chi_{\Omega_k}$ and
$f_j^{(2)}:=f_j\chi_{(\Omega_k)^\com}$. Then
it is easy to see that there exists a positive constant
$C_{(p,q,r)}$, depending on $p,\,q,\,r$, but independent
of $\{f_j\}_j$, such that
\begin{align}\label{five6}
\lf\|\lf\{\sum_{j}\lf[M_{{\rm HL}}(f_j)\r]^r\r\}^
{\frac1r}\r\|_{L^{p,q}(\rn)}
&\leq C_{(p,q,r)}\lf[
\lf\|\lf\{\sum_{j}\lf[M_{{\rm HL}}(f_j^{(1)})\r]^r\r\}^
{\frac1r}\r\|_{L^{p,q}(\rn)}\r.\\
&\hs+\lf.\lf\|\lf\{\sum_{j}\lf[M_{{\rm HL}}(f_j^{(2)})\r]^r\r\}^
{\frac1r}\r\|_{L^{p,q}(\rn)}\r]\noz\\
&=:C_{(p,q,r)}\lf({\rm I_1}+{\rm I_2}\r).\noz
\end{align}

For ${\rm I_1}$, by \eqref{se6}, Lemma \ref{fivel3}(ii),
the fact
$\Omega_k=\bigcup_iQ_i^k$,
\eqref{five5}, \eqref{five4} and \cite[Remark 4.8]{lyy15}, we have
\begin{align}\label{five2}
\lf({\rm I_1}\r)^q
&\sim\sum_{k\in\zz}2^{kq}
\lf|\lf\{x\in\rn:\ \lf[\sum_{j}\lf\{M_{{\rm HL}}(f_j^{(1)})\r\}^r(x)\r]^
{\frac1r}>2^k\r\}\r|^{\frac qp}\\
&\ls\sum_{k\in\zz}2^{kq}2^{-\frac{kq}p}
\lf\|\lf[\sum_{j}\lf|f_j^{(1)}\r|^r\r]^
{\frac1r}\r\|_{L^1(\rn)}^{\frac qp}
\ls\sum_{k\in\zz}2^{kq}2^{-\frac{kq}p}\lf\{\int_
{\bigcup_iQ_i^k}
f(x)\,dx\r\}^{\frac qp}\noz\\
&\ls\sum_{k\in\zz}2^{kq}2^{-\frac{kq}p}\lf[\sum_i
2^k\lf|Q_i^k\r|\r]^{\frac qp}
\ls\sum_{k\in\zz}2^{kq}|\Omega_k|^{\frac qp}
\ls\|M_d(f)\|_{L^{p,q}(\rn)}^q\noz\\
&\ls\|M_{{\rm HL}}(f)\|_{L^{p,q}(\rn)}^q
\ls\|f\|_{L^{p,q}(\rn)}^q.\noz
\end{align}

For ${\rm I_2}$, take $m\in\nn$ satisfying that $mr>p$. Then, by
\eqref{se6} and Lemma \ref{fivel3}(i), we find that
\begin{align}\label{five8}
({\rm I_2})^q
&\sim\sum_{k\in\zz}2^{kq}
\lf|\lf\{x\in\rn:\ \lf[\sum_{j}\lf\{M_{{\rm HL}}(f_j^{(2)})\r\}^r(x)\r]^
{\frac1r}>2^k\r\}\r|^{\frac qp}\\
&\ls\sum_{k\in\zz}2^{kq}2^{-kmr\frac qp}\lf[\int_{\rn}\lf\{\sum_{j}
\lf[M_{{\rm HL}}(f_j^{(2)})\r]^r(x)\r\}^{{\frac1r}mr}\,dx\r]^{\frac qp}\noz\\
&\ls\sum_{k\in\zz}2^{kq}2^{-kmr\frac qp}\lf[\sum_{\ell\in\zz}
2^{\ell mr}\lf|\lf\{x\in\lf(\Omega_k\r)^\com:\ f(x)>2^\ell\r\}\r|\r]^{\frac qp}.\noz
\end{align}
Next we estimate ${\rm I_2}$ by considering two cases: $q/p\in(0,1]$ and $q/p\in(1,\fz]$.

\emph{Case 1:} $q/p\in(0,1]$. For this case,
by \eqref{five8}, the fact that $f(x)\le M_d(f)(x)\le2^k$
for almost every $x\in(\Omega_k)^\com$, $mr>p$ and \eqref{se6}, we know that
\begin{align}\label{five10}
({\rm I_2})^q
&\ls\sum_{k\in\zz}2^{kq}2^{-kmr\frac qp}
\sum_{\ell\in(-\fz,k-1]\cap\zz}2^{\ell mr\frac qp}
\lf|\lf\{x\in\lf(\Omega_k\r)^\com:\ f(x)>2^\ell\r\}\r|^{\frac qp}\\
&\ls\sum_{\ell\in\zz}\sum_{k\in[\ell+1,\fz)\cap\zz}
2^{q(\ell-k)(\frac{mr}p-1)}2^{\ell q}
\lf|\lf\{x\in\lf(\Omega_k\r)^\com:\ f(x)>2^\ell\r\}\r|^{\frac qp}\noz\\
&\ls\sum_{\ell\in\zz}2^{\ell q}\lf|\lf\{x\in\rn:\ f(x)>2^\ell\r\}\r|^{\frac qp}
\sim\|f\|_{L^{p,q}(\rn)}^q.\noz
\end{align}

\emph{Case 2:} $q/p\in(1,\fz]$. For this case, let
$\delta:=\frac{mr-p}2$. Then,
by \eqref{five8}, the H\"{o}lder inequality, the fact that
$f(x)\le M_d(f)(x)\le2^k$ for almost every $x\in(\Omega_k)^\com$,
$mr-\delta>p$ and \eqref{se6} again, we have
\begin{align*}
({\rm I_2})^q
&\ls\sum_{k\in\zz}2^{kq}2^{-kmr\frac qp}2^{k\delta\frac qp}
\sum_{\ell\in(-\fz,k-1]\cap\zz}2^{\ell(mr-\delta)\frac qp}
\lf|\lf\{x\in\lf(\Omega_k\r)^\com:\ f(x)>2^\ell\r\}\r|^{\frac qp}\\
&\ls\sum_{\ell\in\zz}\sum_{k\in[\ell+1,\fz)\cap\zz}
2^{q(\ell-k)(\frac{mr-\delta}p-1)}2^{\ell q}
\lf|\lf\{x\in\lf(\Omega_k\r)^\com:\ f(x)>2^\ell\r\}\r|^{\frac qp}\noz\\
&\ls\sum_{\ell\in\zz}2^{\ell q}\lf|\lf\{x\in\rn:\ f(x)>2^\ell\r\}\r|^{\frac qp}
\sim\|f\|_{L^{p,q}(\rn)}^q,\noz
\end{align*}
which, combined with \eqref{five6}, \eqref{five2} and
\eqref{five10}, implies \eqref{five22}.
This prove (i).

Now we prove (ii). For any $p\in(0,\fz)$, $r\in(1,\fz)$ and $s\in(0,\min\{r,p\})$,
by \eqref{se22} and (i), we know that
\begin{align*}
\lf\|\lf\{\sum_{j}\lf[M_{{\rm HL}}(f_j)\r]^
{\frac rs}\r\}^{\frac1r}\r\|_{L^{p,q}(\rn)}
&=\lf\|\lf\{\sum_{j}\lf[M_{{\rm ML}}
f_j\r]^{\frac rs}\r\}^{\frac sr}\r\|
^{\frac1s}_{L^{p/s,q/s}(\rn)}\\
&\ls\lf\|\lf[\sum_{j}| f_j|^
{\frac rs}\r]^{\frac sr}\r\|
^{\frac1s}_{L^{p/s,q/s}(\rn)}
\sim\lf\|\lf[\sum_{j}\lf|f_j\r|^
{\frac rs}\r]^{\frac1r}\r\|_{L^{p,q}(\rn)}.
\end{align*}
This finishes the proof of (ii) and hence Lemma \ref{fivel6}.
\end{proof}

\begin{definition}\label{fived2}
Let $r,\,\lambda\in(0,\fz)$. For any sequence
$\{s_Q\}_{Q\in\widehat{\mathcal{Q}}}\subset \mathbb{C}$,
its \emph{majorant sequence}
$s_{r,\,\lambda}^*:=
\{(s_{r,\,\lambda}^*)_Q\}_{Q\in\widehat{\mathcal{Q}}}$,
is defined by setting, for all $Q\in\widehat{\mathcal{Q}}$,
$$\lf(s_{r,\,\lambda}^*\r)_Q:=
\lf\{\sum_{P\in\widehat{\mathcal{Q}},\,|P|=|Q|}
\frac{|s_P|^r}{\lf[1+|Q|^{-1}\rho(x_Q-x_P)\r]^
\lambda}\r\}^{\frac1r}.$$
\end{definition}

The following lemma is just \cite[Lemma 6.2]{bh06}.

\begin{lemma}\label{fivel7}
Let $a\in(0,\fz)$, $r\in [a,\fz)$, $\lambda\in(r/a,\fz)$
and $i,\,j\in\mathbb{Z}$. Then there exists a positive constant $C$,
depending only on $\lambda-r/a$, such that, for any sequence
$s:=\{s_P\}_{P\in\widehat{\mathcal{Q}}},\,
Q\in\widehat{\mathcal{Q}}$ with
$|Q|=|\det A|^{-j}$ and $x\in Q$,
it holds true that
\begin{align*}
\lf\{\sum_{|P|=|\det A|^{-i}}\r.
&\lf.|s_P|^r\lf[1+\frac{\rho(x_Q-x_P)}
{\max\lf\{|P|,|Q|\r\}}\r]^{-\lambda}\r\}^{\frac1r}\\
&\leq C|\det A|^{\frac{\max\{0,\,i-j\}}{a}}
\lf\{\lf[M_{{\rm HL}}\lf(\sum_{|P|
=|\det A|^{-i}}|s_P|^a\chi_P\r)\r](x)\r\}^{\frac1a}
\end{align*}
and, in particular, if $i=j$, then
\begin{align*}
\sum_{|Q|
=|\det A|^{-j}}\lf(s_{r,\,\lambda}^*\r)_Q\chi_Q
\leq C\lf[M_{{\rm HL}}\lf(\sum_{|Q|
=|\det A|^{-j}}|s_Q|^a\chi_Q\r)\r]^{\frac1a}.
\end{align*}
\end{lemma}

\begin{lemma}\label{fivel8}
Let $p\in(0, \fz)$ and $q\in(0, \fz]$.
Then, for any $r\in(0, \fz)$ and
$\lambda\in(\max\{1, r/2, r/p\}, \fz)$,
 there exists a positive constant $C$ such that, for all
$s:=\{s_Q\}_{Q\in\widehat{\mathcal{Q}}}$,
\begin{align*}
\lf\|\lf[\sum_{Q\in\widehat{\mathcal{Q}}}
\lf\{\lf(s^*_{r,\lambda}\r)_Q\r\}^2\chi_Q\r]^
{\frac12}\r\|_{L^{p,q}(\rn)}
\leq C\lf\|\lf[\sum_{Q\in\widehat{\mathcal{Q}}}
|s_Q|^2\chi_Q\r]^{\frac12}\r\|_{L^{p,q}(\rn)}
\end{align*}
\end{lemma}

\begin{proof}
Let $r\in(0,\fz)$ and
$\lambda\in(\max\{1,r/2,r/p\},\fz)$.
Choose $a$ such that
$r/\lambda<a<\min\{r, 2, p\}$. Then
$0<a<r<\fz,\ \lambda>r/a,\
2/a>1$ and $p/a>1$.
Therefore, by Lemma \ref{fivel7}, \eqref{se22}
and Lemma \ref{fivel6}(i), we find that
\begin{align*}
&\lf\|\lf[\sum_{Q\in\widehat{\mathcal{Q}}}
\lf\{\lf(s^*_{r,\lambda}\r)_Q\r\}^2\chi_Q\r]^
{\frac12}\r\|_{L^{p,q}(\rn)}\\
&\hs=\lf\|\lf\{\sum_{j\in\mathbb{Z}}\lf[
\sum_{|Q|=|\det A|^{-j}}\lf(s^*_{r,\lambda}\r)_
Q\chi_Q\r]^2\r\}^{\frac12}\r\|_{L^{p,q}(\rn)}\noz\\
&\hs\ls\lf\|\lf(\lf\{\sum_{j\in\mathbb{Z}}
\lf[M_{{\rm HL}}\lf(\sum_{|Q|=|\det A|^{-j}}
|s_Q|^a\chi_Q\r)\r]^{\frac2a}\r\}^{\frac a2}\r)^
{\frac1a}\r\|_{L^{p,q}(\rn)}\noz\\
&\hs\sim\lf\|\lf\{\sum_{j\in\mathbb{Z}}
\lf[M_{{\rm HL}}\lf(\sum_{|Q|=|\det A|^{-j}}
|s_Q|^a\chi_Q\r)\r]^{\frac2a}\r\}^{\frac a2}\r\|^
{\frac1a}_{L^{p/a,q/a}(\rn)}\noz\\
&\hs\ls\lf\|\lf[\sum_{j\in\mathbb{Z}}
\lf(\sum_{|Q|=|\det A|^{-j}}|s_Q|^a\chi_Q\r)^
{\frac2a}\r]^{\frac12}\r\|_{L^{p,q}(\rn)}
\sim\lf\|\lf[\sum_{Q\in\widehat{\mathcal{Q}}}
|s_Q|^2\chi_Q\r]^{\frac12}\r\|_{L^{p,q}(\rn)},\noz
\end{align*}
which completes the proof of Lemma \ref{fivel8}.
\end{proof}

The following Lemma \ref{fivel9}
comes from \cite[Lemma 3.8]{lfy14} (see also \cite[p.\,423]{mb07}).

\begin{lemma}\label{fivel9}
For all $f\in\cs'(\rn)$ and $\Phi\in\cs(\rn)$ with
$\supp\widehat{\Phi}$ being compact
and bounded away from the origin,
the sequences
$\sup(f):=\{\sup_Q(f)\}_
{Q\in\widehat{\mathcal{Q}}}$
and
$\inf(f):=\{\inf_Q(f)\}_
{Q\in\widehat{\mathcal{Q}}}$
are defined by setting, respectively, for any
$Q\in\widehat{\mathcal{Q}}$ with
$|Q|=|\det A|^{-j}$ for some $j\in\mathbb{Z}$,
$${\rm sup}_Q(f):=\sup_{y\in Q}
|f\ast\widetilde{\Phi}_{-j}(y)|$$
and
$${\rm inf}_Q(f):=\sup\lf\{\inf_{y\in P}
|f\ast\widetilde{\Phi}_{-j}(y)|:\ |P\cap Q|\neq0,
\,|Q|/|P|=|\det A|^\gamma\r\},$$
where $\gamma\in\mathbb{N}$.
Then, for all $\lambda,\,r\in(0,\fz)$ and
sufficient large $\gamma\in\mathbb{N}$,
there exists a positive constant $C$ such that,
for any $Q\in\widehat{\mathcal{Q}}$,
$$({\rm sup}_Q(f))_{r,\,\lambda}^*
\leq C({\rm inf}_Q(f))_{r,\,\lambda}^*.$$
\end{lemma}

Now we prove Theorem \ref{fivet2}.

\begin{proof}[Proof of Theorem \ref{fivet2}]
We first prove the necessity of Theorem \ref{fivet2}.
Let $p\in(0,1],\,q\in(0,\fz]$ and $f\in H_A^{p,q}(\rn)$.
By Proposition \ref{fivep1}, we know $f\in\cs'_0(\rn)$.
Furthermore, by \cite[Lemma 2.20]{lfy14} and repeating the proof
of the necessity of Theorem \ref{fivet1} with a slight
modification, we easily conclude that $g(f)\in L^{p,q}(\rn)$
and
$\lf\|g(f)\r\|_{L^{p,q}(\rn)}
\ls\|f\|_{H_A^{p,q}(\rn)}$.
This finishes the proof of the necessity of Theorem \ref{fivet2}.

Thus, to complete the proof of Theorem \ref{fivet2},
it remains to show the sufficiency of Theorem \ref{fivet2}.
To this end,
by Theorem \ref{fivet1}, we only need to prove that
\begin{align}\label{five43}
\lf\|S(f)\r\|_{L^{p,q}(\rn)}
\ls\lf\|g(f)\r\|_{L^{p,q}(\rn)}
\end{align}
for all $f\in\cs'_0(\rn)$ with $g(f)\in L^{p,q}(\rn)$.
We prove this by two steps.

\emph{Step 1.} In this step, we prove that,
for any $r\in(0,1]$, $j\in\mathbb{Z}$, $x\in\rn$
and $y\in x+B_{-j}$,
\begin{align}\label{five3}
\lf|f\ast\varphi_{-j}(y)\r|
&\ls\sum_{i\in\mathbb{Z}}b^
{-|j-i|[1+(s+1)\zeta_--\frac 1r]}\\
&\hs\times\lf\{\lf[M_{{\rm HL}}
\lf(\sum_{Q\in \mathcal{Q}_i}
\lf|\lf(f\ast\widetilde{\Phi}_{-i}\r)(x_Q)\r|^r
\chi_Q\r)\r](x)\r\}^{1/r},\noz
\end{align}
where $\varphi\in\cs(\rn)$
has the vanishing moments up to $\ell$ which
will be fixed later.

To this end,
let $\Phi,\,\Psi\in\cs(\rn)$ be as in
Lemma \ref{fivel5}. For any $M\in\mathbb{N}$,
$j\in\mathbb{Z},\,f\in\cs'_0(\rn),\,x\in\rn$ and
$y\in x+B_{-j}$, by Lemmas \ref{fivel5} and
\ref{fivel4}, we have
\begin{align*}
&\lf|f\ast\varphi_{-j}(y)\r|\\
&\hs\ls\sum_{i\in\mathbb{Z}}\sum_{Q\in
\mathcal{Q}_{i}}b^{-i}\lf|f\ast\widetilde
{\Phi}_{-i}(x_Q)\Psi_{-i}\ast\varphi_{-j}(y-x_Q)\r|\\
&\hs\ls\lf[\sum_{i\leq j}\sum_{Q\in\mathcal{Q}_{i}}
\frac{b^{-(j-i)(\ell+1)\zeta_-}}
{\lf[1+\rho(A^{i}(y-x_Q))\r]^M}+\sum_{i>j}\sum_{Q\in\mathcal{Q}_{i}}
\frac{b^{-(i-j)[1+(\ell+1)\zeta_-]}}
{\lf[1+\rho(A^j(y-x_Q))\r]^M}\r]
\lf|\lf(f\ast\widetilde{\Phi}_{-i}\r)(x_Q)\r|\\
&\hs=:C_{11}({\rm I}+{\rm II}),
\end{align*}
where $C_{11}$ is a positive constant depending on
$\ell$ and $M$, but independent of $j$, $f$ and $y$.

We first estimate ${\rm I}$. Assume that $i\leq j$.
For any $k\in\mathbb{Z}_+$,
$x\in\rn$ and $y\in x+B_{-j}$, when $k=0$, let
$U_0:=\{Q\in\mathcal{Q}_i:\ \rho(A^i(y-x_Q))\leq 1\}$
and, when $k\in\mathbb{N}$, let
$$U_k:=\{Q\in\mathcal{Q}_i:\ b^{k-1}
<\rho(A^i(y-x_Q))\leq b^k\}.$$
Then we have
\begin{align}\label{five44}
&\sum_{Q\in U_k}\frac{b^{-(j-i)
(\ell+1)\zeta_-}}{\lf[1+\rho(A^i(y-x_Q))\r]^M}
\lf|\lf(f\ast\widetilde{\Phi}_{-i}\r)(x_Q)\r|\\
&\hs\ls b^{-(j-i)(\ell+1)\zeta_--kM}\sum_{Q\in U_k}
\lf|\lf(f\ast\widetilde{\Phi}_{-i}\r)(x_Q)\r|\noz.
\end{align}

Notice that, for any $z\in Q\in U_k$, by Definition
\ref{sd2}(iii) and \eqref{five35}, we know that
\begin{align*}
\rho(z-y)\leq H\lf[\rho(z-x_Q)+\rho(x_Q-y)\r]
\leq H\lf(b^{j_0-i}+b^{k-i}\r)<2Hb^{j_0+k-i},
\end{align*}
where $j_0\in\mathbb{Z}_+$ is as in \eqref{five35},
which implies that
\begin{align}\label{five46}
\bigcup_{Q\in U_k}Q
\subset B_\rho(y,2Hb^{j_0+k-i}):=
\lf\{z\in\rn:\ \rho(y-z)<2Hb^{j_0+k-i}\r\}.
\end{align}
Moreover, noticing that $i\leq j$ and $k,\,j_0\in\mathbb{Z}_+$,
for all $x\in y+B_{-j}$, we have
$x\in B_\rho(y,2Hb^{j_0+k-i})$.
Therefore,
for all $r\in(0,1]$ and $x\in y+B_{-j}$,
by \eqref{five46}, we conclude that
\begin{align}\label{five47}
&\sum_{Q\in U_k}
\lf|\lf(f\ast\widetilde{\Phi}_{-i}\r)(x_Q)\r|\\
&\hs\ls b^{\frac{k}r}\lf\{\frac1{\lf|{B_\rho(y,2Hb^{j_0+k-i})}\r|}\int_
{B_\rho(y,2Hb^{j_0+k-i})}\sum_{Q\in U_k}
\lf|\lf(f\ast\widetilde{\Phi}_{-i}\r)(x_Q)\r|^r
\chi_Q(z)\,dz\r\}^{1/r}\noz\\
&\hs\ls b^{\frac{k}r}\lf\{\lf[
M_{{\rm HL}}\lf(\sum_{Q\in \mathcal{Q}_i}
\lf|\lf(f\ast\widetilde{\Phi}_{-i}\r)(x_Q)\r|^r
\chi_Q\r)\r](x)\r\}^{1/r}.\noz
\end{align}
We choose $M>1/r$. Then, by \eqref{five44}
and \eqref{five47}, we find that
\begin{align*}
{\rm I}&\ls\sum_{i\leq j}
\sum_{k=0}^{\fz}b^{-(j-i)(\ell+1)
\zeta_--k(M-\frac 1r)}
\lf\{\lf[M_{{\rm HL}}\lf(\sum_
{Q\in \mathcal{Q}_i}\lf|\lf(f\ast
\widetilde{\Phi}_{-i}\r)(x_Q)\r|^r
\chi_Q\r)\r](x)\r\}^{1/r}\\
&\sim\sum_{i\leq j}
b^{-(j-i)(\ell+1)\zeta_-}
\lf\{\lf[M_{{\rm HL}}\lf(\sum_
{Q\in \mathcal{Q}_i}\lf|\lf(f\ast
\widetilde{\Phi}_{-i}\r)(x_Q)\r|^r
\chi_Q\r)\r](x)\r\}^{1/r}.
\end{align*}

On the other hand, similar to the estimate of ${\rm I}$, by choosing
$M,\,\ell\in\mathbb{N}$ such that $M>1/r$ and
$1+(\ell+1)\zeta_--1/r>0$,
we also obtain
\begin{align*}
{\rm II}
&\ls\sum_{i>j}\sum_{k=0}^{\fz}
b^{-(i-j)[1+(\ell+1)\zeta_--\frac 1r]-
k(M-\frac 1r)}\\
&\hs\times\lf\{\lf[M_{{\rm HL}}\lf(\sum_{Q\in \mathcal{Q}_i}
\lf|\lf(f\ast\widetilde{\Phi}_{-i}\r)(x_Q)\r|^r
\chi_Q\r)\r](x)\r\}^{1/r}\\
&\sim\sum_{i>j}b^{-(i-j)
[1+(\ell+1)\zeta_--\frac 1r]}\lf\{\lf[
M_{{\rm HL}}
\lf(\sum_{Q\in \mathcal{Q}_i}
\lf|\lf(f\ast\widetilde{\Phi}_{-i}\r)(x_Q)\r|^r
\chi_Q\r)\r](x)\r\}^{1/r}.
\end{align*}

Combining the above estimates of ${\rm I}$
and ${\rm II}$, we further conclude that
\eqref{five3} holds true.

\emph{Step 2.} In this step,
we show \eqref{five43} via \eqref{five3}.
Indeed, by \eqref{five3}, we know that,
for all $x\in\rn$,
\begin{align}\label{five48}
\lf[S(f)(x)\r]^2
&=\sum_{j\in\zz}b^j\int_{x+B_{-j}}
\lf|f\ast\varphi_{-j}(y)\r|^2\,dy\\
&\ls\sum_{j\in\mathbb{Z}}
\Bigg\{\sum_{i\in\mathbb{Z}}
b^{-|j-i|[1+(\ell+1)\zeta_--\frac 1r]}
\noz\\
&\lf.\hs\times\lf\{\lf[M_{{\rm HL}}\lf(\sum_
{Q\in \mathcal{Q}_i}\lf|\lf(f\ast
\widetilde{\Phi}_{-i}\r)(x_Q)\r|^r
\chi_Q\r)\r](x)\r\}^{1/r}\r\}^2.\noz
\end{align}
Noticing that $1+(\ell+1)\zeta_--1/r>0$,
by the H\"{o}lder inequality, we have
\begin{align*}
&\sum_{i\in\mathbb{Z}}b^{-|j-i|[1+(\ell+1)
\zeta_--\frac 1r]}\lf\{\lf[M_{{\rm HL}}
\lf(\sum_{Q\in \mathcal{Q}_{i}}\lf|\lf(f\ast
\widetilde{\Phi}_{-i}\r)(x_Q)\r|^r\chi_Q\r)
\r](x)\r\}^{\frac1r}\\
&\hs\ls\lf\{\sum_{i\in\mathbb{Z}}b^{-|j-i|
[1+(\ell+1)\zeta_--\frac 1r]}\lf\{\lf[M_{{\rm HL}}
\lf(\sum_{Q\in \mathcal{Q}_i}\lf|\lf(f\ast
\widetilde{\Phi}_{-i}\r)(x_Q)\r|^r\chi_Q\r)
\r](x)\r\}^{\frac2r}\r\}^{\frac12},
\end{align*}
which, combined with \eqref{five48},
 further implies that
\begin{align}\label{five49}
S(f)(x)\ls\lf\{\sum_{i\in\mathbb{Z}}\lf\{\lf[
M_{{\rm HL}}\lf(\sum_{Q\in \mathcal{Q}_i}
\lf|\lf(f\ast\widetilde{\Phi}_{-i}\r)(x_Q)\r|^r
\chi_Q\r)\r](x)\r\}^{\frac2r}\r\}^{\frac12}.
\end{align}

Choose $M\in\mathbb{N}$ large enough such that
$r\in(1/M,p)$. Then, by \eqref{five49} and
Lemma \ref{fivel6}(i), we find that
\begin{align}\label{five50}
\ \ \ \ \lf\|S(f)\r\|_{L^{p,q}(\rn)}
&\ls\lf\|\lf\{\sum_{i\in\mathbb{Z}}\lf[M_{{\rm HL}}
\lf(\sum_{Q\in\mathcal{Q}_i}\lf|\lf(f\ast
\widetilde{\Phi}_{-i}\r)(x_Q)\r|^r\chi_Q\r)\r]
^{\frac2r}\r\}^{\frac12}\r\|_{L^{p,q}(\rn)}\\
&\hs\ls\lf\|\lf\{\sum_{i\in\mathbb{Z}}\sum_
{Q\in \mathcal{Q}_i}\lf[\lf|\lf(f\ast
\widetilde{\Phi}_{-i}\r)(x_Q)\r|\chi_Q\r]^2\r\}
^{\frac12}\r\|_{L^{p,q}(\rn)}.\noz
\end{align}
Recall that $s_Q\leq(s^*_{r,\,\lambda})_Q$
for any sequence
$\{s_Q\}_{Q\in\widehat{Q}}\subset\mathbb{C}$,
and $r,\,\lambda\in(0,\fz)$. From this,
\eqref{five50}, Lemmas \ref{fivel9} and
\ref{fivel8} with $r\in(0,\fz),\,\lambda\in
(\max\{1, r/2, r/p\},\fz)$,
we deduce that, for some
$\gamma\in\mathbb{N}$ large enough as in
Lemma \ref{fivel9},
\begin{align}\label{five51}
\lf\|S(f)\r\|_{L^{p,q}(\rn)}
&\ls\lf\|\lf\{\sum_{Q\in\widehat{\mathcal{Q}}}
\lf[({\rm sup}_{Q}(f))^*_{r,\,\lambda}\r]^2
\chi_Q\r\}^{\frac12}\r\|_{L^{p,q}(\rn)}\\
&\ls\lf\|\lf\{\sum_{Q\in\widehat{\mathcal{Q}}}
\lf[({\rm inf}_{Q}(f))^*_{r,\,\lambda}\r]^2
\chi_Q\r\}^{\frac12}\r\|_{L^{p,q}(\rn)}\noz\\
&\ls\lf\|\lf\{\sum_{Q\in\widehat{\mathcal{Q}}}
\lf[{\rm inf}_{Q}(f)\r]^2\chi_Q\r\}^{\frac12}\r\|
_{L^{p,q}(\rn)}.\noz
\end{align}
Moreover, for any $P\in\widehat{\mathcal{Q}}$
with $|P|=b^{-i}$ and
$s_P:=\inf_{y\in P}|f\ast
\widetilde{\Phi}_{i-\gamma}(y)|$,
it follows from \cite[p.\,306]{lfy14} that
$\inf_Q(f)=\sup\{s_P: P\in\widehat{\mathcal{Q}},
\,P\cap Q\neq\emptyset,\,|Q|/|P|=b^\gamma\}$ and,
for all $x\in\rn$,
$$\sum_{|Q|=b^{-j}}{\rm inf}_Q(f)\chi_Q(x)\ls
b^{\frac{\gamma\lambda}r}\sum_{|P|=b^
{-j-\gamma}}(s^*_{r,\,\lambda})_P\chi_P(x)$$
(see also the proof of
\cite[Lemma 8.4]{mb07} for more details). By this, \eqref{five51} and Lemma \ref{fivel8}, we have
\begin{align*}
\lf\|S(f)\r\|_{L^{p,q}(\rn)}
&\ls\lf\|\lf[\sum_{j\in\mathbb{Z}}
\sum_{|P|=b^{-j-\gamma}}
\lf|(s^*_{r,\,\lambda})_P\r|^2
\chi_P\r]^{\frac12}\r\|_{L^{p,q}(\rn)}\\
&\ls\lf\|\lf[\sum_{i\in\mathbb{Z}}
\sum_{|P|=b^{-i}}\lf|s_P\r|^2\chi_P\r]^
{\frac12}\r\|_{L^{p,q}(\rn)}\ls\lf\|g(f)\r\|
_{L^{p,q}(\rn)},\noz
\end{align*}
which is \eqref{five43}.
This finishes the proof of the sufficiency of Theorem \ref{fivet2}
and hence Theorem \ref{fivet2}.
\end{proof}

To prove Theorem \ref{fivet3}, we first
recall some notation from \cite{lfy14}. For each
open subset $E\subset\rn$ and $k_0\in\nn$,
let \begin{align}\label{five52}
U_{k_0}:=\lf\{x\in\rn:\ M_{{\rm HL}}
(\chi_E)(x)>b^{-2\tau-k_0}\r\},
\end{align}
where $M_{{\rm HL}}$ denotes the Hardy-Littlewood
maximal function defined by \eqref{te58}
and $\tau\in\mathbb{Z}_+$ is as in \eqref{se1}.
For any $\varphi\in\cs(\rn)$, $k_0\in\zz_+$, $f\in\cs'(\rn)$
and $x\in\rn$, let
$$S_{k_0}(f)(x):=\lf\{\sum_{k\in\mathbb{Z}}
b^{-(k+k_0)}\int_{x+B_{k+k_0}}
\lf|f\ast\varphi_k(y)\r|^2\,dy\r\}^{\frac12},$$
which is a variant of the anisotropic Lusin-area
function $S$ defined by \eqref{five1}.
Obviously, $S=S_0$.

The proof of the following Lemma \ref{fivel10} is
similar to that of \cite[Lemma 3.12]{lfy14},
the details being omitted.

\begin{lemma}\label{fivel10}
There exists a positive
constant $C$ such that, for all $k_0\in\mathbb{N}$ and
$f\in\cs'(\rn)$,
\begin{align*}
\int_{(U_{k_0})^\complement}\lf[S_{k_0}(f)(x)\r]^2\,dx
\leq C\int_{E^\complement}\lf[S(f)(x)\r]^2\,dx,
\end{align*}
where $E\subset\rn$ is an
open set and $U_{k_0}$
is as in \eqref{five52}.
\end{lemma}

The following technical lemma plays a key role in
proving Theorem \ref{fivet3}, whose proof
is motivated by Folland and Stein
\cite[Theorem 7.1]{fs82}, Aguilera and
Segovia \cite[Theorem 1]{as77}, Liang et al.
\cite[Lemma 4.6]{lhy12} and Li et al.
\cite[Lemma 3.13]{lfy14}.

\begin{lemma}\label{fivel11}
Let $p\in(0,1]$ and $q\in(0,\fz]$. Then there exists
a positive constant $C$ such that,
for all $k_0\in\mathbb{N}$ and $f\in L^{p,q}(\rn)$,
$$\lf\|S_{k_0}(f)\r\|_{L^{p,q}(\rn)}\leq Cb^
{(\frac1p-\frac12)k_0}\lf\|S(f)\r\|_{L^{p,q}(\rn)}.$$
\end{lemma}

\begin{proof}
For any $k_0\in\mathbb{N}$, $\lambda\in(0,\fz)$
and $f\in L^{p,q}(\rn)$,
let
$$E_{\lambda,k_0}:=\lf\{x\in\rn:\ S(f)(x)
>\lambda b^{k_0/2}\r\}$$
and
$$U_{\lambda,k_0}:=\lf\{x\in\rn:\ M_{{\rm HL}}(\chi_
{E_{\lambda,k_0}})(x)>b^{-2\tau-k_0}\r\},$$
where $M_{{\rm HL}}$ is as in \eqref{te58}.
Then, by the boundedness
from $L^1(\rn)$ to $L^{1,\fz}(\rn)$ of $M_{{\rm HL}}$
(see Proposition \ref{tl4} and Remark \ref{tr1}),
we have
$$|U_{\lambda,k_0}|\ls b^{k_0}
\|\chi_{E_{\lambda,k_0}}\|_{L^1(\rn)}
\sim b^{k_0}|E_{\lambda,k_0}|.$$
From this, together with Lemma \ref{fivel10} with
$E=E_{\lambda,k_0}$ and $U_{k_0}=U_{\lambda,k_0}$,
we deduce that
\begin{align}\label{five53}
&\lf|\lf\{x\in\rn: S_{k_0}(f)>\lambda\r\}\r|\\
&\hs\leq\lf|U_{\lambda,k_0}\r|+
\lf|(U_{\lambda,k_0})^\complement\cap
\lf\{x\in\rn:\ S_{k_0}(f)(x)>\lambda\r\}\r|\noz\\
&\hs\ls b^{k_0}\lf|E_{\lambda,k_0}\r|+
\lambda^{-2}\int_{(U_{\lambda,k_0})^
\complement}\lf[S_{k_0}(f)(x)\r]^2\,dx\noz\\
&\hs\ls b^{k_0}\lf|E_{\lambda,k_0}\r|+
\lambda^{-2}\int_{(E_{\lambda,k_0})^
\complement}\lf[S(f)(x)\r]^2\,dx\noz\\
&\hs\sim b^{k_0}\lf|E_{\lambda,k_0}\r|+
\lambda^{-2}\int_0^{\lambda b^{k_0/2}}
\mu\lf|\lf\{x\in\rn:\ S(f)(x)>\mu\r\}\r|\,d\mu.\noz
\end{align}
Moreover, for all $\ell\in\mathbb{Z}$
and $k_0\in\mathbb{N}$, by \eqref{five53}, we have
\begin{align}\label{five55}
&\lf|\lf\{x\in\rn:\ S_{k_0}(f)(x)>2^\ell\r\}\r|\\
&\hs\ls b^{k_0}\lf|E_{2^\ell,k_0}\r|+2^{-2\ell}
\int_0^{2^\ell b^{k_0/2}}\mu\lf|\lf\{x\in\rn:\
S(f)(x)>\mu\r\}\r|\,d\mu\noz\\
&\hs\ls b^{k_0}\lf|E_{2^\ell,k_0}\r|+2^{-2\ell}
\sum_{m=-\fz}^{m_\ell}\int_{2^{m-1}}^{2^m}
\mu\lf|\lf\{x\in\rn:\ S(f)(x)>\mu\r\}\r|\,d\mu\noz\\
&\hs\ls b^{k_0}\lf|E_{2^\ell,k_0}\r|+2^{-2\ell}
\sum_{m=-\fz}^{m_\ell}2^{2(m-1)}\lf|\lf\{x\in\rn:\
 S(f)(x)>2^{m-1}\r\}\r|,\noz
\end{align}
where
$m_\ell:=\ell+\lfloor\frac{k_0}2\log_2b\rfloor+1$.
Next we show the desired conclusion by considering
three cases: $q/p\in(1,\fz)$, $q/p\in(0,1]$ and $q=\fz$.

\emph{Case 1:} $q/p\in(1,\fz)$.
For this case, by \eqref{se6},
\eqref{five53}, the definition of $E_{\lambda,k_0}$
and the H\"{o}lder inequality, we conclude that
\begin{align}\label{five54}
\lf\|S_{k_0}(f)\r\|^q_{L^{p,q}(\rn)}
&\sim\int_0^\fz\lambda^{q-1}\lf|\lf\{x\in\rn:\ S_
{k_0}(f)(x)>\lambda\r\}\r|^{\frac qp}\,d\lambda\\
&\ls\int_0^\fz b^{\frac qpk_0}\lambda^{q-1}
\lf|\lf\{x\in\rn:\ S(f)(x)>\lambda b^{k_0/2}\r\}\r|^
{\frac qp}\,d\lambda\noz\\
&\hs+\int_0^\fz\lambda^{q-\frac{2q}p-1}\lf[\int_
0^{\lambda b^{k_0/2}}\mu\lf|\lf\{x\in\rn:\ S(f)(x)
>\mu\r\}\r|\,d\mu\r]^{\frac qp}\,d\lambda\noz\\
&\ls\int_0^\fz b^{(\frac1p-\frac12)qk_0}\lambda^
{q-1}\lf|\lf\{x\in\rn:\ S(f)(x)>\lambda\r\}\r|^{\frac qp}
\,d\lambda\noz\\
&\hs+\int_0^\fz\lambda^{q-\frac{2q}p-1}\lf[\int_
0^{\lambda b^{k_0/2}}\mu^{(1-p)\frac q{q-p}}
\,d\mu\r]^{\frac qp-1}\noz\\
&\hs\times\int_0^{\lambda b^{k_0/2}}\mu^q\lf|\lf
\{x\in\rn:\ S(f)(x)>\mu\r\}\r|^{\frac qp}\,d\mu
\,d\lambda\noz\\
&\sim b^{(\frac1p-\frac12)qk_0}\lf\|S(f)\r\|^q_{L^
{p,q}(\rn)}+b^{(\frac1p-\frac12)qk_0-\frac {k_0}2}\noz\\
&\hs\times\int_0^\fz\lambda^{-2}\int_0^{\lambda
b^{k_0/2}}\mu^q\lf|\lf\{x\in\rn:\ S(f)(x)>\mu\r\}\r|^
{\frac qp}\,d\mu\,d\lambda\noz\\
&\ls b^{(\frac1p-\frac12)qk_0}\lf\|S(f)\r\|^q_{L^
{p,q}(\rn)}+b^{(\frac1p-\frac12)qk_0}\noz\\
&\hs\times\int_0^\fz\mu^{q-1}\lf|\lf\{x\in\rn:\ S(f)(x)
>\mu\r\}\r|^{\frac qp}\,d\mu\noz\\
&\sim b^{(\frac1p-\frac12)qk_0}\lf\|S(f)\r\|^
q_{L^{p,q}(\rn)}\noz.
\end{align}

\emph{Case 2:}
$q/p\in(0,1]$. For this case, by \eqref{se6}, \eqref{five55}
and the definition of $E_{\lambda,k_0}$, we find that
\begin{align}\label{five56}
\lf\|S_{k_0}(f)\r\|^q_{L^{p,q}(\rn)}
&\sim\sum_{\ell\in\mathbb{Z}}2^{\ell q}\lf|\lf\{
x\in\rn:\ S_{k_0}(f)(x)>2^{\ell}\r\}\r|^{\frac qp}\\
&\ls\sum_{\ell\in\mathbb{Z}}2^{\ell q}
b^{\frac qpk_0}\lf|\lf\{x\in\rn:\ S(f)(x)>2^{\ell}
b^{k_0/2}\r\}\r|^{\frac qp}\noz\\
&\hs+\sum_{\ell\in\mathbb{Z}}2^{\ell q}2^{-\frac
{2q}p\ell}\sum_{m=-\fz}^{m_\ell}2^{2(m-1)\frac qp}
\lf|\lf\{x\in\rn:\ S(f)(x)>2^{m-1}\r\}\r|^{\frac qp}\noz\\
&\ls b^{(\frac1p-\frac12)qk_0}\lf\|S(f)\r\|^
q_{L^{p,q}(\rn)}+\sum_{\ell\in\mathbb{Z}}2^{-\ell q}
b^{(\frac1p-1)qk_0}\noz\\
&\hs\times\sum_{m=-\fz}^{m_\ell}2^{2q(m-1)}
\lf|\lf\{x\in\rn:\ S(f)(x)>2^{m-1}\r\}\r|^{\frac qp}\noz\\
&\sim b^{(\frac1p-\frac12)qk_0}\lf\|S(f)\r\|^q_{L^{p,q}(\rn)}
+\sum_{m\in\mathbb{Z}}2^{-\ell q}b^{(\frac1p-1)qk_0}\noz\\
&\hs\times\sum_{\ell=\ell_m}^\fz2^{2q(m-1)}\lf|\lf
\{x\in\rn:\ S(f)(x)>2^{m-1}\r\}\r|^{\frac qp}\noz\\
&\ls b^{(\frac1p-\frac12)qk_0}\lf\|S(f)\r\|^q_{L^{p,q}(\rn)}
+b^{(\frac1p-\frac12)qk_0}\noz\\
&\hs\times\sum_{m\in\mathbb{Z}}2^{q(m-1)}\lf|\lf
\{x\in\rn:\ S(f)(x)>2^{m-1}\r\}\r|^{\frac qp}\noz\\
&\sim b^{(\frac1p-\frac12)qk_0}\lf\|S(f)\r\|
^q_{L^{p,q}(\rn)},\noz
\end{align}
where $\ell_m:=m-\lfloor\frac{k_0}2\log_2b\rfloor-1$.

\emph{Case 3:}
For the last case when $q=\fz$, by \eqref{se7},
\eqref{five53} and the definition of $E_{\lambda,k_0}$ again,
we know that, for any $\lambda\in(0,\fz)$,
\begin{align}\label{five57}
&\lambda\lf|\lf\{x\in\rn:\ S_{k_0}(f)(x)>\lambda\r\}\r|
^{\frac1p}\\
&\hs\ls b^{\frac{k_0}p}\lambda\lf|\lf\{x\in\rn:\ S(f)(x)
>b^{\frac{k_0}2}\lambda\r\}\r|^{\frac1p}\noz\\
&\hs\hs+\lambda^{1-\frac2p}\lf\{\int_0^{\lambda b^{k_
0/2}}\mu\lf|\lf\{x\in\rn:\ S(f)(x)>\mu\r\}\r|\,d\mu\r\}
^{\frac1p}\noz\\
&\hs\ls\lf\|S(f)\r\|_{L^{p,\fz}(\rn)}\lf[b^{(\frac1p-
\frac12)k_0}
+\lambda^{1-\frac2p}\lf(\int_0^{\lambda b
^{k_0/2}}\mu^{1-p}\,d\mu\r)^{\frac1p}\r]\noz\\
&\hs\sim b^{(\frac1p-\frac12)k_0}\lf\|S(f)\r\|
_{L^{p,\fz}(\rn)}.\noz
\end{align}
Combining \eqref{five54}, \eqref{five56} and
\eqref{five57}, we complete the proof of Lemma \ref{fivel11}.
\end{proof}

Now we prove Theorem \ref{fivet3}.

\begin{proof}[Proof of Theorem \ref{fivet3}]
Let $p\in(0,1],\,q\in(0,\fz]$ and $\lambda\in(2/p,\fz)$.
Notice that, for all $f\in\cs'_0(\rn)$ with
$g_\lambda^*(f)\in L^{p,q}(\rn)$,
$S(f)\ls g_\lambda^*(f)$. Then, by Theorem \ref{fivet1},
we know that
$$\lf\|f\r\|_{H_A^{p,q}(\rn)}
\ls\|g_\lambda^*\|_{L^{p,q}(\rn)}.$$

Conversely, for all $f\in H_A^{p,q}(\rn)$, by Proposition
\ref{fivep1}, we know $f\in\cs'_0(\rn)$. Thus, to complete the proof
of Theorem \ref{fivet3},
by Theorem \ref{fivet1} again, we only need to show that,
for all $f\in H_A^{p,q}(\rn)$,
\begin{align}\label{five58}
\lf\|g_\lambda^*(f)\r\|_{L^{p,q}(\rn)}
\ls\lf\|S(f)\r\|_{L^{p,q}(\rn)}.
\end{align}
Indeed, for all $x\in\rn$,
\begin{align*}
\lf[g_\lambda^*(f)(x)\r]^2
&=\sum_{k\in\mathbb{Z}}b^{-k}
\int_{x+B_k}\lf[\frac{b^k}{b^k+\rho(x-y)}\r]
^\lambda\lf|f\ast\varphi_{k}(y)\r|^2\,dy\\
&\hs+\sum_{m=1}^\fz\sum_{k\in\mathbb{Z}}
b^{-k}\int_{x+(B_{k+m}\setminus B_{k+m-1})}
\cdots\\
&\ls\lf[S(f)(x)\r]^2+\sum_{m=1}^\fz
b^{-m(\lambda-1)}\lf[S_m(f)(x)\r]^2.
\end{align*}
From this, the Aoki-Rolewicz theorem (see \cite{ta42, sr57}),
Lemma \ref{fivel11} and $\lambda\in(2/p,\fz)$,
we further deduce that there exists $\upsilon\in(0,1]$
such that
\begin{align*}
\lf\|g_\lambda^*(f)\r\|^{\upsilon}_{L^{p,q}(\rn)}
&\ls\sum_{m=0}^\fz b^{-m\upsilon(\lambda-1)/2}
\lf\|S_m(f)\r\|^{\upsilon}_{L^{p,q}(\rn)}\\
&\ls\sum_{m=0}^\fz b^{-m\upsilon(\lambda-1)/2}
b^{(\frac1p-\frac12)m\upsilon}\lf\|S(f)\r\|^{\upsilon}
_{L^{p,q}(\rn)}\ls\lf\|S(f)\r\|^{\upsilon}_{L^{p,q}(\rn)},
\end{align*}
which implies that
$\lf\|g_\lambda^*(f)\r\|_{L^{p,q}(\rn)}
\ls\lf\|S(f)\r\|_{L^{p,q}(\rn)}$.
This finishes the proof of \eqref{five58}
and hence Theorem \ref{fivet3}.
\end{proof}

\textbf{Acknowledgements.} The authors would like to express their
deep thanks to Professor Qixiang Yang for several useful conversations
on Lemma \ref{fivel6} which is an important tool of this article
and is of independent interest.

\bigskip

\noindent  Jun Liu, Dachun Yang (Corresponding author) and Wen Yuan

\medskip

\noindent  School of Mathematical Sciences, Beijing Normal University,
Laboratory of Mathematics and Complex Systems, Ministry of
Education, Beijing 100875, People's Republic of China

\smallskip

\noindent {\it E-mails}: \texttt{junliu@mail.bnu.edu.cn} (J. Liu)

\hspace{1.12cm}\texttt{dcyang@bnu.edu.cn} (D. Yang)

\hspace{1.12cm}\texttt{wenyuan@bnu.edu.cn} (W. Yuan)

\end{document}